\documentclass[10pt]{amsart}
\usepackage{graphicx} 
\usepackage{xcolor}
\usepackage{setspace} 
\usepackage{amsfonts,amssymb,latexsym,amsmath,amsthm}
\usepackage{enumerate} 
\usepackage{ulem} 
\usepackage{currfile} 
\usepackage[margin=1in]{geometry}
\usepackage{dsfont} 
\usepackage{bbm} 
\newtheorem{theorem}{Theorem}[section]
\newtheorem{definition}[theorem]{Definition}

\newtheorem{corollary}[theorem]{Corollary}
\newtheorem{lemma}[theorem]{Lemma}

\numberwithin{equation}{section} 
\begin{document}
\title[ZCP wavelets]{Wavelets for $L^2(B(0, 1))$ using Zernike polynomials}
\author{Somantika Datta}
\address{Department of Mathematics and Statistical Science, University of Idaho, ID 83844-1103, USA}
\email{sdatta@uidaho.edu}
%
\author{Kanti B. Datta}
\address{Department of Electrical Engineering, Indian Institute of Technology, Kharagpur, 721302, India}
\email{attadbitnak@yahoo.com}

%
%
%
%
%
\maketitle
\doublespacing
\begin{abstract}
 A set of orthogonal polynomials on the unit disk $B(0,1)$ known as Zernike polynomials are commonly used in the analysis and evaluation of optical systems. Here Zernike polynomials are used to construct localized bases and \textit{wavelet functions} for polynomial subspaces of $L^2(B(0,1)).$ This naturally leads to a multiresolution analysis of  $L^2(B(0,1)).$ Previously, other authors have dealt with the one dimensional case, and used orthogonal polynomials of a single variable to construct time localized bases for polynomial subspaces of an $L^2$-space with arbitrary weight. Due to the nature of Zernike polynomials, the wavelet decomposition that results here is well-suited for the analysis of two-dimensional signals defined on circular domains. This is shown by some experimental results done on corneal data.  
\end{abstract}

%


\section{Introduction}
\label{sec:Intro}
%
%
%
\subsection{Background} \label{Subsec:Background}
Wavelets can be viewed as a tool to obtain expansions of functions in Hilbert spaces. Traditionally, a wavelet system is made of a function $\psi$ and all integer translations and dilations of $\psi:$
    $$
    \psi_{j,k} =2^{j/2} \psi(2^j x - k), \quad x \in \mathbb{R}, \ j,k \in \mathbb{Z},
    $$
such that  $\{\psi_{j,k}\}_{j,k \in \mathbb{Z}}$ is an orthonormal basis (ONB) of $L^2(\mathbb{R}).$ The first example of such a $\psi$ was given by Haar in 1910 \cite{Haar1910}, and is called the Haar wavelet.  \textit{Multiresolution analysis} (MRA) is a tool to construct wavelets, and is a central ingredient in wavelet analysis. The general framework of MRA was devised  by Mallat \cite{Mallat89} and Meyer \cite{Meyer}. An MRA consists of a nested sequence $\{V_j\}_{j \in \mathbb{Z}}$ of closed subspaces of $L^2(\mathbb{R}),$ giving a decomposition of $L^2(\mathbb{R}).$ In wavelet analysis, a signal $f$ is decomposed into pieces, each subspace $V_j$ has a piece of $f.$ These pieces, or projections, of $f$ give finer and finer details of $f.$ In connection with approximation theory, approximation of $f$ in $L^2(\mathbb{R})$ can be performed via the projection of $f$ on one of these subspaces.


When a signal $f$ is represented in terms of a wavelet basis $\{\psi_{j,k}\}_{j,k \in \mathbb{Z}}$, then it can be written as
    \begin{equation} \label{EQ:WaveletRepresentation}
    f(x) = \sum_{j, k} c_{j,k} \psi_{j,k} (x) . 
    \end{equation}
A representation such as (\ref{EQ:WaveletRepresentation}) gives useful information about when or where a certain frequency occurs in the signal $f.$ Different frequencies appear with different values of $j$ and this fact is useful in time-frequency or space-frequency analysis. Slow oscillations of $f$ will lead to nonzero coefficients for small values of $j,$ whereas fast oscillations will lead to nonzero coefficients for large values of $j.$ The location of a frequency is indicated by the corresponding $k.$ Such information regarding the location of a frequency is not evident in the Fourier series or Fourier-Zernike series ((\ref{EQ:ZernikeFourierExp}) below) of a signal.

A set of two dimensional orthogonal functions defined on the unit disk $B(0,1)$, called \textit{Zernike polynomials}, is used in the analysis of optical systems by expanding optical wavefront functions as series of these functions \cite{Zernike1934, Nijboer1942}. The Zernike polynomials form an orthonormal basis for $L^2(B(0,1)),$ the space of square integrable functions on the unit disc, and can thus be used to effectively represent signals on circular domains \cite{Zernike1934, Nijboer1942}.
In this work, Zernike polynomials have been used to construct \textit{localized bases} for polynomial subspaces of $L^2(B(0,1)).$ This is done using kernel polynomials made from Zernike polynomials and exploiting the property that kernel polynomials are localized. This has been inspired by previous work done in \cite{Fischer1997} for the 1D case where the authors have used orthogonal polynomials to construct time-localized bases for polynomial subspaces of an $L^2$-space with arbitrary weight. The construction in \cite{Fischer1997} is based on the general theory of kernel polynomials \cite{Szego1975}, and this has been used there to get a multiresolution analysis (MRA) of a weighted $L^2$-space that is different from traditional MRA. Similarly, we will be able to get a multiresolution of $L^2(B(0,1))$ by decomposing into subspaces that are each spanned by Zernike polynomials of certain degrees only (indicating signal frequencies). These subspaces are called \textit{wavelet spaces}.
In \cite{Kilgore1996}, a decomposition of the space $L^2(-1, 1)$ has been investigated using wavelets from algebraic polynomials using Chebyshev-weight.  Even though multidimensional wavelets have been constructed and studied by some researchers (see, for example, \cite{Madych1992, GM1992}), such analysis is suited for rectangular domains, and the resulting wavelets are meant for functions on $\mathbb{R}^d,$ $d \geq 2.$ The wavelet analysis that is proposed here using Zernike polynomials is suitable for signals on circular domains  like that of optical data, and can be useful in the space-frequency analysis of such data. The primary motivation is efficient representation of 2D signals defined on circular domains, such as corneal surfaces and certain optical systems; wavelet functions constructed from Zernike polynomials aim at detecting the location of abonormalities and may be used for characterizing aberrations. 
We have implemented our 2D wavelet functions in the reconstruction of corneal data from its wavelets coefficients, and used the information of the wavelet coefficients to study location of spatial frequency in the data. Our main goal here is the development of theoretical results, and the numerical experiments have been done to simply demonstrate the theory.

\subsection{Preliminaries and Notation} \label{Subsec:P&N}

Let $\mathbb{N}_0 = \{ 0, 1, 2, 3, \ldots\} = \mathbb{N} \cup \{0\}.$  The \textit{Zernike polynomials} are defined in terms of complex exponentials as
    \begin{equation} \label{EQ:ZCP}
    Z_n^m(r, \phi)
    =
    \gamma_{n}^m R_n^{|m|}(r)e^{i m \phi} , \quad 0\leq r \leq 1, \ 0 \leq \phi < 2 \pi,
    \end{equation}
where $n \in \mathbb{N}_0,$ $|m| \leq n,$ $n - m$ is even, $\gamma_{n}^m$ is the normalization constant given by
    $$
    \gamma_{n}^m =
     \sqrt{\frac{n+1}{\pi}} \quad \forall m,
    $$
$R_n^{|m|}$ are the \textit{radial parts} of the polynomials given by
    \begin{equation} \label{EQ:RadialPolySeries}
    R_n^{|m|}(r) = \sum_{s = 0}^{(n - |m|)/2} \frac{(-1)^s(n-s)!}{s! ((n+|m|)/2 - s)!((n - |m|)/2 - s)!} r^{n - 2s} ,
    \end{equation}
  and will be referred to as \textit{radial polynomials}. 
When a single index notation is needed, the conversion from $Z_{n}^m$ to $Z_j,$  $j \in \mathbb{N}_0,$ is made by the formula
    \begin{equation} \label{EQ:ZCPIndex}
    j = \frac{n(n+2) + m}{2}.
    \end{equation}
It is known that the Zernike polynomials $\{Z^m_n(r, \phi)\}_{\substack{n = 0,|m| \leq n \\ (n - m) \textrm{even}}}^{\infty}$ given in (\ref{EQ:ZCP}) form a complete orthonormal set \cite{Nijboer1942} for $L^2(B(0,1)),$ the space of square integrable functions on the unit disk $B(0,1) = \{(x,y) \in \mathbb{R}^2 : x^2 + y^2 \leq 1\},$ with respect to the inner product
    \begin{equation} \label{EQ:InnerProduct}
    \langle f, g \rangle : = \int_0^1 \int_0^{2\pi} f(r, \phi) \overline{g(r, \phi )} \ r \ d\phi \ d r.
    \end{equation}
  In polar form, the Zernike polynomials of (\ref{EQ:ZCP}) can be written  as
    $$
    G^m_n(r, \phi)
    =
    \left\{
    \begin{array}{cc}
    \gamma^m_n R^{|m|}_n (r) \cos (m \phi) , & \textrm{if $m \geq 0,$} \\
    \gamma^m_n R^{|m|}_n (r) \sin (|m| \phi) , & \textrm{if $m < 0,$}
    \end{array}
    \right.
    $$
where $n \in \mathbb{N}_0,$ $|m| \leq n,$ $n - m$ is even, $\gamma^m_n$ are normalization constants given by
    $$
    \gamma_{n}^m = \left\{
    \begin{array}{cc}
     \sqrt{\frac{n+1}{\pi}} & \textrm{if} \ m = 0 \\
     \sqrt{\frac{2(n+1)}{\pi}} & \textrm{if} \ m \neq 0,
     \end{array}
    \right.
    $$
and $R_n^{|m|}$ are radial polynomials already given in (\ref{EQ:RadialPolySeries}). Each $G^m_n(r, \phi)$ is a polynomial in $x$, $y$ of degree $n$, and $\{G^m_n(r, \phi)\}_{\substack{n = 0,|m| \leq n \\ (n - m) \textrm{even}}}^{\infty}$ form a complete orthonormal set in $L^2(B(0, 1))$  with respect to the same inner product given above in (\ref{EQ:InnerProduct}) . Any $f \in L^2(B(0, 1))$ can be written in terms of $Z^m_n$s or $G^m_n$s as follows.
    \begin{eqnarray}
    f(\rho, \theta)
    &=&
    \sum_{n=0}^{\infty} \sum_{\substack{m = 0 \\ (n - m) \textrm{even}}}^n  \left[ A_{nm}\cos m\theta + B_{nm} \sin m\theta \right] R^m_n(\rho) \nonumber \\
    &=&
    \sum_{n=0}^{\infty} \sum_{\substack{m = -n \\ (n - m) \textrm{even}}}^n  c_{nm} e^{i m \theta} R^{|m|}_n (\rho). \label{EQ:ZernikeFourierExp}
    \end{eqnarray}
The connection between the corresponding coefficients is given by the following. For all $n \in \mathbb{N}_0,$ $(n - m)$ even, $m \leq n,$
    \begin{eqnarray*}
    c_{nm} &=& \frac{A_{nm} - i B_{nm}}{2} , \quad m \in \mathbb{N}_0, \\
    c_{n(-m)} &=& \frac{A_{nm} + i B_{nm}}{2} , \quad m \in \mathbb{N},
    \end{eqnarray*}
and for all $n, m \in \mathbb{N}_0,$ $(n - m)$ even, $m \leq n,$
    \begin{eqnarray*}
    A_{nm} &=& c_{nm} + c_{n(-m)}, \\
    B_{nm} &=& i (c_{nm} - c_{n(-m)}).
    \end{eqnarray*}
  A series as in (\ref{EQ:ZernikeFourierExp}) is called the \textit{Fourier-Zernike series} of $f.$
    
For a given pair $(\rho, \theta)$ with $0 \leq \rho \leq 1$ and $0 \leq \theta < 2\pi,$ the polynomial
    \begin{equation} \label{EQ:KernelPolynomials}
    K_N(r, \phi; \rho, \theta) := \sum_{n=0}^{N} \sum_{\substack{m=-n \\ (n-m) \textrm{even}}}^n \overline{Z^m_n(\rho, \theta)} Z_n^m (r, \phi)
    \end{equation}
is called the $N$th \textit{kernel polynomial} with respect to the inner product in (\ref{EQ:InnerProduct}) and the parameters $\rho,$ $\theta$. See \cite{Szego1975} for the general theory of kernel polynomials.

\noindent
For $\alpha = (\alpha_1, \alpha_2) \in \mathbb{N}_0^2$ define
    $$
    |\alpha| := \alpha_1 + \alpha_2,
    $$
and for $x = (x_1, x_2)$ in $\mathbb{R}^2$,
    $$
    x^{\alpha} := x_1^{\alpha_1} x_2 ^{\alpha_2}.
    $$
Let
    $$
    V_N := \textrm{span}\{x^{\alpha} : |\alpha| \leq N, x_1^2 + x_2^2 \leq 1\},
    $$
i.e., $V_N$ is the space of all polynomials in two variables of degree at most $N$ defined on the unit disk $B(0, 1).$ Here $N \in \mathbb{N}_0.$ The dimension of $V_N$ is $\frac{(N+1)(N+2)}{2}$ \cite{YuanXu2004}. 
If $N=0$, $V_N$ is of dimension 1 and is the space generated by $\{1\},$ i.e., the constant functions. In fact, $V_0 = \textrm{span}\{\mathds{1}_{B(0,1)} \},$ where $\mathds{1}_{B(0,1)}$ is the characteristic function of $B(0,1).$

%
Note that if $p \in V_N$ then
    \begin{equation} \label{EQ:RepPropKernelPoly}
    \langle p, K_N(.,.; \rho, \theta)\rangle = \sum_{n=0}^N \sum_{\substack{m=-n \\ (n-m) \textrm{even}}}^n \langle p, Z^m_n \rangle Z^m_n (\rho, \theta) =  p(\rho, \theta) .
    \end{equation}
Thus the kernel polynomial $K_N$ has the above reproducing property for the space $V_N.$ As mentioned earlier, the notion of a kernel polynomial is the main tool used to construct localized bases for the polynomial subspaces $V_N.$ That the kernel polynomial $K_N(r, \phi; \rho, \theta)$ is localized around $(\rho, \theta)$ follows from Lemma~\ref{LEM:OptimizationProblemSolution} below.

\subsection{Outline}\label{Subsec:Outline}
In Section~\ref{SEC:ScalingFunctions}, scaling functions using Zernike polynomials are defined for the space $V_N$ and various properties of the scaling functions are discussed. Wavelet functions and duals of wavelets are presented in Section~\ref{SEC:Wavelets} and Section~\ref{SEC:Duals}, respectively. 
A multiresolution analysis of $L^2(B(0,1))$ using the spaces $V_N$ and the scaling functions of Section~\ref{SEC:ScalingFunctions} is discussed in Section~\ref{SEC:MRA}. Finally, numerical results demonstrating the theory are given in Section~\ref{SEC:NumResults}.

\section{Scaling functions} \label{SEC:ScalingFunctions}
\noindent
Fix $N \in \mathbb{N}.$ In this section, scaling functions for $V_N$ are defined using kernel polynomials made of Zernike polynomials. Some useful properties of scaling functions are presented. Each scaling function for $V_N$ is parametrized by a point $(\rho, \theta)$ in $B(0, 1),$ and is localized around this point (Lemma~\ref{LEM:OptimizationProblemSolution}).  This important property is what enables us to get a localized basis out of a set of scaling functions. 
\begin{lemma}\label{LEM:CirclePolyONB}
The set $\{R^{|m|}_n(r) e^{i m \phi}\}_{\substack{n = 0,|m| \leq n \\ (n - m) \ \textrm{even}}}^N$ \footnote{To keep the notation less cumbersome we shall not always write that $m - n$ is even in the subscript.}  is an orthogonal basis for $V_N.$
\end{lemma}
\begin{proof}
Note that $\{R^{|m|}_n(r) e^{i m \phi}\}_{\substack{n = 0,|m| \leq n \\ (n - m) \ \textrm{even}}}^{N}$ or $\{R^m_n(r) \cos m \phi, R^m_n(r) \sin m \phi \}_{\substack{n = 0, 0 \leq m \leq n \\ (n - m) \ \textrm{even}}}^N$  consists of polynomials in $x$, $y$ of degree up to $N$ \cite{Nijboer1942}. It is also known that this is an orthogonal set. Since the dimension of $V_N$ is $\binom{N + 2}{N} = \frac{(N+1)(N+2)}{2}$ \cite{YuanXu2004}, we just need to show that there are $\frac{(N+1)(N+2)}{2}$ elements in $\{R^m_n(r) e^{i m \phi}\}_{n=0, |m| \leq n}^N$, $(n-m)$ even. This will be shown by considering the cases for $N$ being odd and even.

When $N$ is odd, the number of radial polynomials is
    \begin{eqnarray*}
    N_{RO} = 2 + 4 + \cdots + (N+1) = 2(1 + 2 + \cdots + \frac{N+1}{2}) = \frac{(N+1)(N+3)}{4} .
    \end{eqnarray*}
The number of polynomials in $\{R^m_n(r) \cos m \phi, R^m_n(r) \sin m \phi \}_{n=0}^N,$ $0 \leq m \leq n,$ $(n-m)$ even, will then be given by $2 N_{RO}$ minus the number of radial polynomials for $m = 0$ i.e., the number of polynomials is
    $$
    2N_{RO} - \frac{N+1}{2} = \frac{(N+1)(N+2)}{2}
    $$
as needed.

When $N$ is even, the number of radial polynomials is
    \begin{eqnarray*}
    N_{RE} &=& 1 + 3 + 5 + \cdots + (N+1) = (1 + 2 + 3 + \cdots + N+1) - (2 + 4 + \cdots + N) \\
    &=& \frac{(N+1)(N+2)}{2} - 2(1 + 2 + \cdots + \frac N2) 
    = \frac{(N+1)(N+2)}{2} - \frac{N(N+2)}{4} = \left( \frac{N+2}{2} \right)^2.
    \end{eqnarray*}
In this case, the number of polynomials in $\{R^m_n(r) \cos m \phi, R^m_n(r) \sin m \phi \}_{n=0}^N,$ $0 \leq m \leq n,$ $(n-m)$ even, will be given by
    $$
    2 N_{RE} - \left( \frac N2 + 1 \right) = \frac{(N+1)(N+2)}{2}.
    $$
\end{proof}
%
%
%

\noindent
A direct consequence of Lemma~\ref{LEM:CirclePolyONB} is the following. 
\begin{corollary} \label{COR:ONB_PI_N}
The set $\{Z^m_n\}_{\substack{n = 0,|m| \leq n \\ (n - m) \textrm{even}}}^N$ is an orthonormal basis for $V_N.$
\end{corollary}

\noindent
Fix $\rho \in [0,1)$, and $\theta \in [0, 2\pi).$ The following Lemma~\ref{LEM:OptimizationProblemSolution} indicates that the kernel polynomials defined in (\ref{EQ:KernelPolynomials}) are localized around the point $(\rho, \theta).$ 
\begin{lemma}\label{LEM:OptimizationProblemSolution}
For given $\rho,$ $\theta,$ consider the following optimization problem
$$
\min \{ \| p \| \ : p \in V_N, \ p(\rho, \theta) = 1 \} .
$$
The solution to this is given by $\left\|\frac{K_N(r, \phi; \rho, \theta)}{K_N(\rho, \theta; \rho, \theta)} \right\|$ where $K_N(r, \phi; \rho, \theta)$ is the kernel polynomial defined in (\ref{EQ:KernelPolynomials}).
\end{lemma}
\begin{proof}
Let $p(r, \phi) \in V_N$ with $p(\rho, \theta) = 1.$ Then
    $$
    p(r, \phi) = \sum_{n = 0}^N \sum_{\substack{m=-n \\ (m - n) \ \textrm{even}}}^n c_{nm} Z^m_n(r, \phi) .
    $$
Thus
    $$
    \| p \|^2 = \langle p, p \rangle = \sum_{n = 0}^N \sum_{\substack{m=-n \\ (m - n) \ \textrm{even}}}^n |c_{nm}|^2 . 
    $$
Using Cauchy-Schwarz we get the following:
    \begin{eqnarray}
    1
    &=&
    | p(\rho, \theta) |^2
    \leq
    \left( \sum_{n=0}^N \sum_{\substack{m=-n \\ (m - n) \ \textrm{even}}}^n |c_{nm}|^2 \right) \left( \sum_{n = 0}^N \sum_{\substack{m=-n \\ (m - n) \ \textrm{even}}}^n |Z^m_n(\rho, \theta)|^2 \right)\nonumber \\
    \textrm{or,} \ 1
    &\leq&
    \| p \|^2 \sum_{n=0}^N \sum_{\substack{m=-n \\ (m - n) \ \textrm{even}}}^n (\gamma_n^m)^2 (R^{|m|}_n(\rho))^2 . \label{EQ:KernelPolyNormBound}
    \end{eqnarray}
By the orthonormality of Zernike polynomials we get 
\begin{align} \label{EQ:ZCPNorm}
& \|K_N(r,\varphi;\rho,\theta)\|^2  =\langle K_N(r,\varphi;\rho,\theta),K_N(r,\varphi;\rho,\theta) \rangle \nonumber \\
& =\left \langle \sum_{n=0}^N \sum_{\substack{m=-n,\\n-m \,\text{even}}}^{m=n}   Z_n^m(r,\varphi) \,  \overline{Z_n^m(\rho,\theta)} , \sum_{n=0}^N \sum_{\substack{m=-n,\\n-m \,\text{even}}}^{m=n}   Z_n^m(r,\varphi) \,  \overline{Z_n^m(\rho,\theta)}\right\rangle \nonumber \\
& = \sum_{n=0}^N \sum_{\substack{m=-n,\\n-m \,\text{even}}}^{m=n} (\gamma_n^m)^2 (R_n^{|m|}(\rho))^2,
\end{align} 
and from (\ref{EQ:KernelPolynomials}) we have 
\begin{equation} \label{EQ:ZCPNorm2}
K_N(\rho, \theta; \rho, \theta) 
= \sum_{n=0}^N \sum_{\substack{m=-n \\ (m - n) \ \textrm{even}}}^n  | Z_n^m(\rho,\theta |^2 =\sum_{n=0}^N \sum_{\substack{m=-n,\\n-m \,\text{even}}}^{m=n} (\gamma_n^m)^2 (R_n^{|m|}(\rho))^2.
\end{equation}
Combining (\ref{EQ:ZCPNorm}), (\ref{EQ:ZCPNorm2}), and (\ref{EQ:KernelPolyNormBound}), we obtain,
    \begin{eqnarray*}
    \left\| \frac{K_N(r, \phi; \rho, \theta)}{K_N(\rho, \theta; \rho, \theta)} \right\|^2
    &=&
    \frac{\| K_N(r, \phi; \rho, \theta) \|^2}{\left(\sum_{n=0}^N \sum_{\substack{m=-n \\ (m - n) \ \textrm{even}}}^n (\gamma^m_n)^2 ({R^{|m|}_n}(\rho))^2 \right)^2}
    = \frac{1}{\sum_{n=0}^N \sum_{\substack{m=-n \\ (m - n) \ \textrm{even}}}^n  (\gamma^m_n)^2 ({R^m_n}(\rho))^2 }
    \leq \| p \|^2 .
    \end{eqnarray*}
\end{proof}
%
%
%

\begin{lemma}\label{LEM:Christoffel-Darboux}
The kernel polynomial $K_N(r,\phi; \rho,\theta)$ satisfies the \textit{Christoffel-Darboux Formula} for Zernike polynomials of radial degree $N$:
\begin{align*}
 K_N(r,\phi;\rho,\theta) &=
\sum_{m=-N}^{m=N}b_N^m \frac{Z_{N+2}^m(r,\phi) \overline{Z_{N}^m(\rho,\theta)} - 
Z_{N}^m(r,\phi)\overline{Z_{N+2}^m(\rho,\theta)}} {(r^2-\rho^2)} \\
& \qquad + \sum_{m=-(N-1)}^{m=N-1}b_{N-1}^m \frac{Z_{N+1}^m(r,\phi) \overline{Z_{N-1}^m(\rho,\theta)} - 
Z_{N-1}^m(r,\phi)\overline{Z_{N+1}^m(\rho,\theta)}} {(r^2-\rho^2)}, 
\end{align*}
where
\begin{alignat*}{3}
b_N^m & =\frac{(N+2)^2-m^2}{4(N+2)}\frac{1}{\sqrt{(N+1)(N+3)}} . 
\end{alignat*}
\end{lemma}
\begin{proof}
To establish this, we use
the following three-term recurrence relation for a set of Zernike polynomials $\{Z_n^m \}$ from \cite{kinter76}:
\begin{alignat}{3}
r^2 Z_n^m(r,\phi) &= b_{n-2} Z_{n-2}^m(r,\phi)+ a_n Z_n^m(r,\phi)+b_{n}Z_{n+2}^m(r,\phi),
\label{eq:CDthree_term7a}
\end{alignat}
where
\begin{alignat*}{3}
a_n & := \left[\frac{(n+m)^2}{n}+\frac{(n-m+2)^2}{n+2} \right]\frac{1}{4(n+1)} . 
\end{alignat*}
Substituting (\ref{eq:CDthree_term7a}) in 
\begin{equation} \label{eq:CDF1}
Z_{n+2}^m(r,\phi) \overline{Z_{n}^m(\rho,\theta)}- Z_{n}^m(r,\phi)\overline{Z_{n+2}^m(\rho,\theta)} ,
\end{equation}
and simplifying with the assumption that n and m are even, it follows that
\begin{align}
 b_n^m \frac{Z_{n+2}^m(r,\phi) \overline{Z_{n}^m(\rho,\theta)} - 
Z_{n}^m(r,\phi)\overline{Z_{n+2}^m(\rho,\theta)}}
{(r^2-\rho^2)} 
&= Z_n^m(r,\phi)\overline{Z_{n}^m(\rho,\theta)} 
 + b_{n-2}^m
\frac{Z_{n}^m(r,\phi)\overline{Z_{n-2}^m(\rho,\theta)}-Z_{n-2}^m(r,\phi) \overline{Z_{n}^m(\rho,\theta)}}{(r^2-\rho^2)} \label{eq:CDF2}  \\
& = Z_n^m(r,\phi)\overline{Z_{n}^m(\rho,\theta)} + 
Z_{n-2}^m(r,\phi)\overline{Z_{n-2}^m(\rho,\theta)}+ \cdots+ 
Z_0^m(r,\phi)\overline{Z_{0}^m(\rho,\theta)} +  \nonumber \\
 & \qquad \qquad \qquad \qquad \qquad + b_{-2}^m \frac{Z_{0}^m(r,\phi)\overline{Z_{-2}^m(\rho,\theta)}-
Z_{-2}^m(r,\phi) \overline{Z_{0}^m(\rho,\theta)}}{(r^2-\rho^2)}  \nonumber \\
& = \sum_{k=0}^{n/2} Z_{n-2k}^m(r,\phi) \overline{Z_{n-2k}^m(\rho,\theta)} \label{eq:CDF3}
\end{align}
since $Z_{-2}^m(\cdot,\cdot)=0$, $|m| \le n$ and $n-|m|$ is even.

Again taking $n$ even, with $n-1$ and $\mu$, an odd number, replacing $n$ and $m$, respectively, in (\ref{eq:CDF1}), and proceeding in the same way as above, we can write, 
\begin{alignat}{3}
& b_{n-1}^{\mu} \frac{Z_{n+1}^{\mu}(r,\phi) \overline{Z_{n-1}^{\mu}(\rho,\theta)} - 
Z_{n-1}^{\mu}(r,\phi)\overline{Z_{n+1}^{\mu}(\rho,\theta)}}
{(r^2-\rho^2)} 
 = Z_{n-1}^{\mu}(r,\phi)\overline{Z_{n-1}^{\mu}(\rho,\theta)}  \label{eq:CDF4} \\
& \qquad \qquad \qquad \qquad \qquad + b_{n-3}^{\mu}
\frac{Z_{n-1}^{\mu}(r,\phi)\overline{Z_{n-3}^{\mu}(\rho,\theta)}-Z_{n-3}^{\mu}(r,\phi) \overline{Z_{n-1}^{\mu}(\rho,\theta)}}{(r^2-\rho^2)}  \notag \\
& = Z_{n-1}^{\mu}(r,\phi)\overline{Z_{n-1}^{\mu}(\rho,\theta)} + 
Z_{n-3}^{\mu}(r,\phi)\overline{Z_{n-3}^{\mu}(\rho,\theta)}+ \cdots+ 
Z_0^{\mu}(r,\phi)\overline{Z_{0}^{\mu}(\rho,\theta)} +  \notag  \\
 &  \qquad \qquad \qquad \qquad \qquad + b_{-2}^{\mu} \frac{Z_{0}^{\mu}(r,\phi)\overline{Z_{-2}^{\mu}(\rho,\theta)}-
Z_{-2}^{\mu}(r,\phi) \overline{Z_{0}^{\mu}(\rho,\theta)}}{(r^2-\rho^2)} \notag  \\
& = \sum_{k=0}^{(n-1-{\mu})/2} Z_{n-1-2k}^{\mu}(r,\phi) \overline{Z_{n-1-2k}^{\mu}(\rho,\theta)} \label{eq:CDF5}
\end{alignat}
since $Z_{-2}^\mu (\cdot,\cdot)=0$, $|\mu| \le n$ and $n-1-|\mu|$ is even. For $n$ odd, a similar calculation can be done. 

Taking the sum of the series (\ref{eq:CDF3}) for $m$ from $-N$ to $N$, and that of (\ref{eq:CDF5}) for $\mu$ from $-(N-1)$ to $N-1$, we get  the kernel polynomial $K_N(r,\phi;\rho,\theta),$ and the corresponding sum of (\ref{eq:CDF2}) and (\ref{eq:CDF4}) gives the right side of the desired Christoffel-Darboux Formula for Zernike polynomials for even radial degree $N$ . For odd radial degree $N,$ things work out in a similar fashion. 

%
\end{proof}

\noindent
Lemma~\ref{LEM:OptimizationProblemSolution} motivates the following definition of scaling functions.

\begin{definition}\label{DEF:SpecialPoints}
(a)
Given $N,$ set $k = \lfloor \frac N2 \rfloor +1 .$
Choose radii $\{ \lambda_i \}_{i = 1}^k$ according to formula (7) of Section 2.3 in \cite{Lopez2016} where
    $$
    1 > \lambda_1 > \lambda_2 > \cdots > \lambda_k \geq 0,
    $$
and on the $i$th circle ($1\leq i \leq k$) choose
    $$
    n_i = 2N + 5 - 4i
    $$
equally spaced nodes. With this choice, there are $J = \frac{(N+1)(N+2)}{2}$ points within the unit circle. Such a set of points in the unit circle will be called \textit{regular points}.

\noindent
(b)
Let $\{(\rho_j^{(N)}, \theta_j^{(N)})\}_{j = 1}^J$ be a set of regular points. The scaling functions are defined as
$$
\phi_{N, j}(r, \phi) := K_N(r, \phi; \rho^{(N)}_{j}, \theta_j^{(N)}), \quad j = 1, \ldots, J.
$$
\end{definition}

%

\noindent
The basis property of the scaling functions is shown in the following Theorem~\ref{THM:BasisTranslatesKernelPoly}. The dual of this basis, which is needed for reconstruction, is given in part (iv) of Theorem~\ref{THM:PropertiesScalingFunction}.

\begin{theorem}\label{THM:BasisTranslatesKernelPoly}
For a given $N,$ let $\{P_j^{(N)} = (\rho_j^{(N)}, \theta_j^{(N)})\}_{j = 1}^J$ be a set of regular points as described in Definition~\ref{DEF:SpecialPoints}. The set
    $$
    \{K_N(r, \phi; \rho_j^{(N)}, \theta_j^{(N)})\}_{j=1}^J
    $$
is a basis of $V_N.$  Consequently, the set of scaling functions $\{ \phi_{N, j}(r, \phi) \}_{j=1}^J $ is a basis of $V_{N}$.

\end{theorem}
\begin{proof}
Let $\{Z_j\}_{j=0}^{J-1}$ denote the Zernike polynomials enumerated using formula (\ref{EQ:ZCPIndex}).
Following results in \cite{Lopez2016}, consider the collocation matrix
    $$
    A_N = [Z_{j-1}(P_i^{(N)})]_{i,j = 1}^J.
    $$
Let $A_N^{(i)}$ be obtained from $A_N$ by replacing the $i$th row with $\{Z_{j-1}(x,y)\}_{j = 1}^J$. Under the choice of points $\{P_j^{(N)} \}_{j = 1}^J,$ the Lagrange interpolating polynomial is uniquely determined \cite{Bos1981}, the $i$th fundamental Lagrange polynomial of degree $N$ is given by
    $$
    \ell_{i}(x, y) = \frac{1}{\textrm{det} A_N} \textrm{det}\left(A_N^{(i)}(x, y)\right),
    $$
and these are characterized by $\ell_i(P_j^{(N)}) = \delta_{ij}.$

Note that the size of the set $\{K_N(r,\phi;\rho_j^{(N)},\theta_j^{(N)} \}_{j=1}^J$ is the same as the dimension of $V_N$. Therefore, to show that this set is a basis of $V_N$, it is enough to show that it is linearly independent.
For convenience of notation, $ \{K_N(r, \phi; \rho_j^{(N)}, \theta_j^{(N)})\}_{j=1}^J $ will be written as $ \{K_N(r, \phi; P_j^{(N)})\}_{j=1}^J.$ Let

    \begin{equation} \label{EQ:KernelPolynomialsLI}
    \sum_{j =1}^J \tau_j K_N(r, \phi; P_j^{(N)}) = 0.
    \end{equation}

For a fixed $i$, take the inner product on both sides of (\ref{EQ:KernelPolynomialsLI}) with the $i$th fundamental Lagrange polynomial $\ell_{i}.$ The reproducing property of the kernel polynomial $K_N$ given in (\ref{EQ:RepPropKernelPoly}) then gives

    $$
    0 = \sum_{j =1}^J \tau_j \left\langle \ell_i, K_N(r, \phi; P_j^{(N)}) \right\rangle = \sum_{j =1}^J \tau_j \ell_i(P_j^{(N)}) = \tau_i.
    $$
Since this holds for each $i = 1, \ldots, J,$ the set is linearly independent.

\end{proof}


\noindent
Following are some properties of scaling functions. For convenience, a point $(\rho^{(N)}_{j}, \theta_j^{(N)})$ will be written as $(\rho_{j}, \theta_j)$ or just $P_j.$

\begin{theorem}\label{THM:PropertiesScalingFunction}

\begin{enumerate}[(i)]
\item The inner product of the scaling functions may be evaluated as
    $$
    \langle \phi_{N,i}, \phi_{N,j}\rangle = \phi_{N,i}(P_j)
    $$

\item The scaling function $\phi_{N,j}$ is localized around $P_j.$ More precisely,
    $$
    \left\|\frac{\phi_{N,j}(r, \phi)}{\phi_{N,j}(P_j)} \right\| = \min \{ \| p \| \ : p \in V_N, \ p(P_j) = 1 \} .
    $$

\item The set of scaling functions $\{\phi_{N,j}\}_{j=1}^J$ is a basis of $V_N.$

\item The dual of the scaling functions $\{\phi_{N,j}\}_{j=1}^J$ is given by the fundamental Lagrange interpolating polynomials $\{\ell_{j}\}_{j=1}^J$ with respect to the points $\{P_j\}_{j=1}^J$ given above.

\item The scaling function $\phi_{N, j}$ is orthogonal to $V_{N-1}$ with respect to the modified inner product \\
$\langle \  \cdot \ , \ \cdot(\cdot - P_j)\rangle $ i.e.,
    $$
    \left \langle \phi_{N, j}(\cdot ,\cdot) \ , \ q(\cdot , \cdot)(\cdot - P_j)\right \rangle = 0 \quad \textrm{for all} \ q \in V_{N-1}.
    $$

\item The scaling function $\phi_{N,j}(r,\phi)$ satisfies the Christoffel-Darboux formula:
\begin{alignat*}{3}
&\phi_{N,j}(r,\phi) 
 = \sum_{m=-N}^{m=N}b_N^m \frac{Z_{N+2}^m(r,\phi) \overline{Z_{N}^m(P_j)} - 
Z_{N}^m(r,\phi)\overline{Z_{N+2}^m(P_j)}} {(r^2-\rho_j^2)} \\
& \qquad \qquad \qquad + \sum_{m=-(N-1)}^{m=N-1}b_{N-1}^m \frac{Z_{N+1}^m(r,\phi) \overline{Z_{N-1}^m(P_j)} - 
Z_{N-1}^m(r,\phi)\overline{Z_{N+1}^m(P_j)}} {(r^2-\rho_j^2)} .
\end{alignat*}

\end{enumerate}

\end{theorem}
\begin{proof}
(i)
$$
\langle \phi_{N,i}, \phi_{N,j}\rangle
=
\langle K_N(r, \phi; P_i), K_N(r, \phi; P_j) \rangle
=
 \phi_{N,i}(P_j)
$$
which is true by the reproducing property of the kernel polynomial $K_N.$

\noindent
(ii) This follows immediately from Lemma~\ref{LEM:OptimizationProblemSolution}.

\noindent
(iii)
This follows immediately from Theorem~\ref{THM:BasisTranslatesKernelPoly}.

\noindent
(iv) The dual $\{\widetilde{\phi}_{N,j}\}_{j=1}^J$ should satisfy the biorthogonality condition, i.e.,
    $$
    \langle \phi_{N,j}, \widetilde{\phi}_{N,k}\rangle = \delta_{jk} .
    $$
Recall the fundamental Lagrange interpolating polynomials $\{\ell_k (x,y)\}_{k=1}^J$ with respect to the points $\{P_j\}_{j=1}^J$ defined in the proof of Theorem~\ref{THM:BasisTranslatesKernelPoly}. Due to the reproducing property of the kernel polynomials we get
    $$
    \langle \phi_{N,j}, \ell_{k} \rangle = \ell_{k}(P_j) =  \delta_{jk} .
    $$
Thus the Lagrange interpolating polynomials $\{\ell_k (x,y)\}_{k=1}^J$ form the dual of the scaling functions.

\noindent
(v) Let $q \in V_{N-1}.$ Applying the reproducing property of the kernel polynomial given in (\ref{EQ:RepPropKernelPoly}) to $q(r,\phi) ((r,\phi) - P_j)$ gives
    $$
    \langle \phi_{N, j}(\cdot ,\cdot) \ , \ q(\cdot , \cdot)(\cdot - P_j) \rangle
    = \overline{q(P_j)} \overline{(P_j - P_j)} = 0.
    $$

 \noindent (vi) This follows from Lemma~\ref{LEM:Christoffel-Darboux}. 
\end{proof}

\begin{theorem}
Consider a set of regular points $\{P_j\}_{j=1}^J$ as described in Definition~\ref{DEF:SpecialPoints}. The following conditions are equivalent regarding the scaling functions $\{\phi_{N,j}\}_{j=1}^J$.

\begin{enumerate} [(i)]

\item The scaling functions form an orthogonal set i.e.,
    $$
    \langle \phi_{N,k}, \phi_{N,\ell} \rangle  = 0 \quad \textrm{for $k \neq \ell$}.
    $$

\item The scaling functions satisfy
    $$
    \phi_{N,k}(P_{\ell}) = d_k^{(N)} \delta_{kl} \quad \textrm{for} \ k, \ell = 1, \ldots, J
    $$
where $d_k^{(N)} \in \mathbb{R}.$

\end{enumerate}
\end{theorem}
\begin{proof}
We first show that (i) implies (ii).
\\
From Theorem~\ref{THM:PropertiesScalingFunction} (i) it is known that
    $$
    \langle \phi_{N,k}, \phi_{N,\ell}\rangle = \phi_{N,k}(P_{\ell}) .
    $$
Assuming (i) holds, this implies that
    $$
    \phi_{N,k}(P_{\ell})
    =
    \left\{
    \begin{array}{cc}
    0 & \textrm{if $k \neq \ell$} \\
    \langle \phi_{N,k}, \phi_{N, k}\rangle = \| \phi_{N,k} \|^2 & \textrm{if $k = \ell.$}
    \end{array}
    \right.
    $$
Setting $d_k^{(N)} = \| \phi_{N,k} \|^2$ for each $k = 1, \ldots, J$ gives the required result.
\\
(ii) implies (i) is obvious by Theorem~\ref{THM:PropertiesScalingFunction} (i).

\end{proof}

\begin{figure}[!h]
\centering
\includegraphics[width=1 \textwidth]{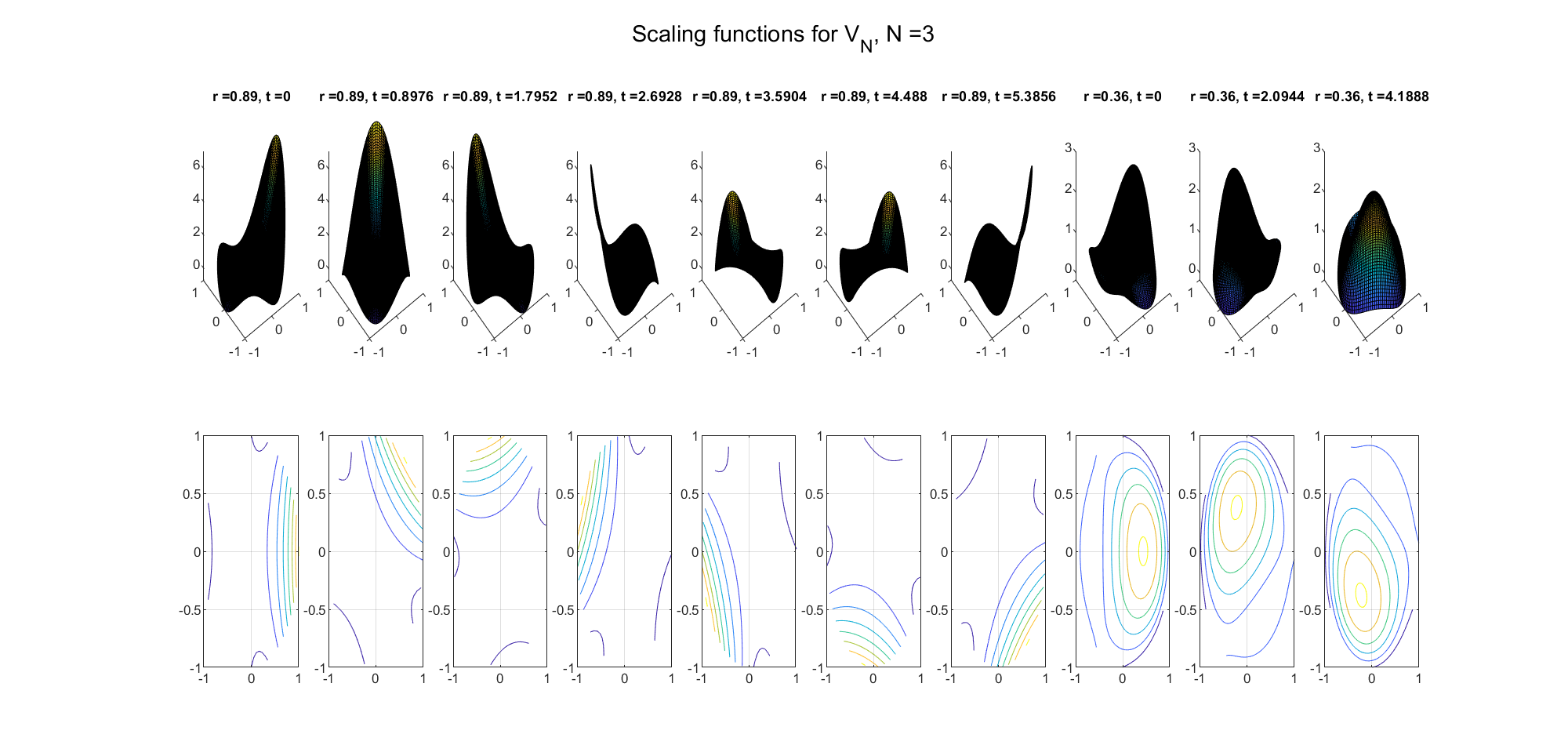}
\caption{The surfaces and contour lines of the scaling functions for $V_3$.}
\label{FIG:ScaleFunctions}
\end{figure}
\begin{figure}[!h]
\centering
\includegraphics[width=.5 \textwidth]{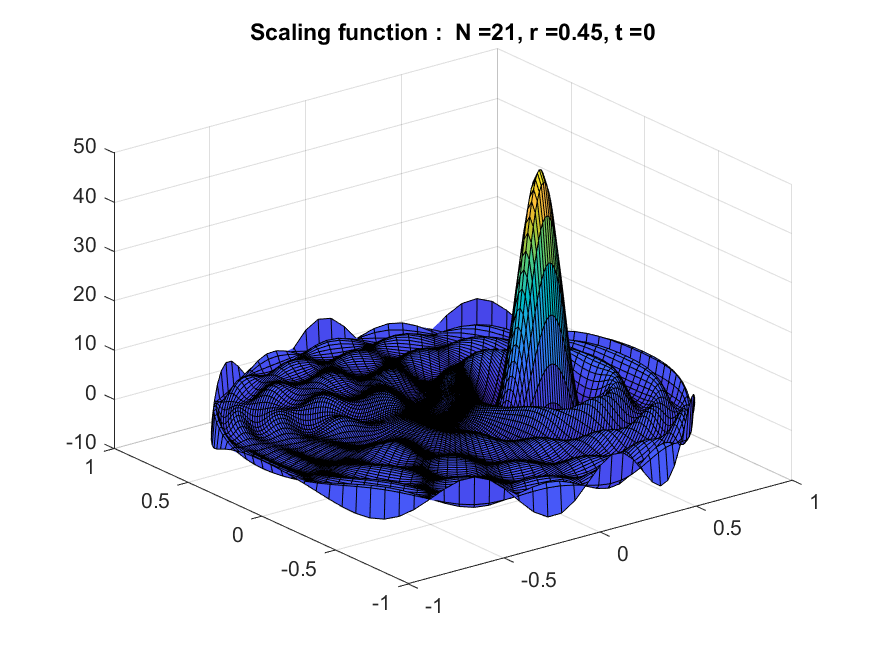}
\caption{The surface of a scaling function for $V_{21}$.}
\label{FIG:ScaleFunctions2}
\end{figure}

The scaling functions for $N = 3$ are shown in Figure~\ref{FIG:ScaleFunctions}. The values of $r$ and $t$ in the figure represent the values of $\rho$ and $\theta,$ respectively, corresponding to each scaling function, with $t$ being in radians and between zero and $2\pi.$ The localized nature of the scaling functions in not evident in Figure~\ref{FIG:ScaleFunctions} due to the fact that the value of the maximum degree $N$ of polynomials considered there is too small. For higher values of $N$, one can clearly see that the scaling functions are localized. This is shown in Figure~\ref{FIG:ScaleFunctions2} for one scaling function when $N = 21$ at $t = 0$ and $r = .45,$ which are the values of $\theta$ and $\rho,$ respectively.

\section{Wavelets}\label{SEC:Wavelets}

In this section, we define wavelet functions in terms of kernel polynomials and present some of their properties. For $N \in \mathbb{N}_0,$ let $J_N = \frac{(N+1)(N+2)}{2}$ denote the dimension of $V_N.$ Consider the sequence of spaces $\{V_N\}_{N=0}^{\infty}$ where, using Corollary~\ref{COR:ONB_PI_N} and  Theorem~\ref{THM:PropertiesScalingFunction} (iii),
    $$
    V_N = \textrm{span} \{Z^m_n\}_{\substack{n = 0, |m| \leq n \\ (n - m) \textrm{even}}}^N = \textrm{span} \{ \phi_{N, j} \}_{j=1}^{J_N}.
    $$
Note that in $\{ \phi_{N, j} \}_{j=1}^{J_N},$ where $\phi_{N, j} = K_N(r, \phi; \rho^{(N)}_{j}, \theta_j^{(N)}),$ we are considering a set of regular points $\{\rho^{(N)}_{j}, \theta_j^{(N)})\}_{j=1}^{J_N}$ as described in Definition~\ref{DEF:SpecialPoints} with $J$ replaced by $J_N$ in the notation. For $N \in \mathbb{N},$ define
    $$
    W_N := V_{2N} \ominus V_N
    =
    \textrm{span}\{ Z^m_n \}_{\substack{n = N+1 , |m| \leq n \\ (n - m) \textrm{even}}}^{2N}
    =
    \textrm{span}\{x^{\alpha} : N < |\alpha| \leq 2N, \alpha \in \mathbb{N}_0^2 \},
    $$
   and
    $$
    W_0 := V_1 \ominus V_0 = \textrm{span}\{x_1, x_2\}. 
    $$
For $N \in \mathbb{N},$ the dimension of $W_N$ is
    $$
    D_N := \textrm{dim} (W_N) = \textrm{dim} (V_{2N})  - \textrm{dim} (V_N) = \frac{(2N+1)(2N+2)}{2} - \frac{(N+1)(N+2)}{2} = \frac{3N(N+1)}{2},
    $$
    and 
    $$
    D_0 := \textrm{dim} (W_0) = \textrm{dim} (V_{1})  - \textrm{dim} (V_0) = 2. 
    $$
Our target functions, called \textit{wavelet functions} or \textit{wavelets} in short, that we wish to form a localized basis for the space $W_N,$ are defined in polar form as
    \begin{align}
    \psi_{N,j}(r, \phi)
    &:=
    K_{2N}(r, \phi; \mu_j^{(N)}, \omega_j^{(N)}) - K_N(r, \phi; \mu_j^{(N)}, \omega_j^{(N)}) \nonumber \\
    &=
    \sum_{n = N+1}^{2N} \sum_{\substack{m=-n \\ (n-m) \textrm{even}}}^n \overline{Z^m_n(\mu_j^{(N)}, \omega_j^{(N)})} Z_n^m (r, \phi), \quad j = 1, \ldots, D_N, \label{EQ:WaveletZCP} \\
    &=
    \sum_{\ell = J_N + 1}^{J_{2N}} \overline{Z_{\ell}(\mu_j^{(N)}, \omega_j^{(N)})} Z_{\ell} (r, \phi), \quad j = 1, \ldots, D_N, \label{EQ_2:WaveletZCP}
    \end{align}
for a suitable set of parameter points $\{\mu^{(N)}_{j}, \omega_j^{(N)})\}_{j=1}^{D_N}.$

\bigskip

\noindent
Some properties of the wavelets defined in (\ref{EQ:WaveletZCP}) are now given.
\begin{theorem} \label{THM:PropertiesZCPWavelets}
Let $\{\psi_{N,j}(r, \phi)\}_{j=1}^{D_N}$ be the wavelets as defined in (\ref{EQ:WaveletZCP}) with respect to the set $\{(\mu^{(N)}_{j}, \omega_j^{(N)})\}_{j=1}^{D_N}.$
Then the following hold.

\begin{enumerate}[(i)]

\item
$$
\langle \psi_{N,j}(r, \phi), \psi_{N,k}(r, \phi) \rangle = \psi_{N,k}(\mu^{(N)}_{j}, \omega_j^{(N)}).
$$

\item The wavelet $\psi_{N,j}$ is localized around $(\mu^{(N)}_{j}, \omega_j^{(N)}),$ i.e.,
    $$
    \left\| \frac{\psi_{N,j} (r, \phi)}{\psi_{N,j}(\mu^{(N)}_{j}, \omega_j^{(N)})} \right\| = \min \{\|\psi\| : \psi \in W_N, \psi(\mu^{(N)}_{j}, \omega_j^{(N)}) = 1\} .
    $$

\item 
Let $\{\phi_{N,j}\}_{j=1}^{J_N}$ be the set of scaling functions in $V_N$ based on the set of points $\{\rho^{(N)}_{j}, \theta_j^{(N)}\}_{j=1}^{J_N}.$ Then the wavelets and scaling functions are orthogonal, i.e.,
        $$
        \langle \psi_{N,j}, \phi_{N,k}\rangle = 0; \quad j = 1, \ldots, D_N, \ k = 1, \ldots, J_N.
        $$

\end{enumerate}

\end{theorem}
\begin{proof}
\begin{enumerate}[(i)]

\item

\begin{align}
\langle \psi_{N,j}(r, \phi), \psi_{N,k}(r, \phi) \rangle
&=
\left \langle K_{2N}(\cdot, \cdot ; \mu_j^{(N)}, \omega_j^{(N)}) - K_N(\cdot, \cdot; \mu_j^{(N)}, \omega_j^{(N)}), \psi_{N,k}(\cdot, \cdot) \right\rangle \nonumber \\
&=
\left\langle K_{2N}(\cdot, \cdot ; \mu_j^{(N)}, \omega_j^{(N)}), \psi_{N,k}(\cdot, \cdot) \right\rangle - \left\langle K_N(\cdot, \cdot; \mu_j^{(N)}, \omega_j^{(N)}), \psi_{N,k}(\cdot, \cdot) \right\rangle . \label{EQ:IPZCPWavelets}
\end{align}
Due to the reproducing property of the kernel polynomials, the first term on the right of (\ref{EQ:IPZCPWavelets}) is equal to $\psi_{N,k}(\mu_j^{(N)}, \omega_j^{(N)}).$ The second term on the right of (\ref{EQ:IPZCPWavelets}) is zero due to orthogonality of the Zernike polynomials. Thus we have
$$
\langle \psi_{N,j}(r, \phi), \psi_{N,k}(r, \phi) \rangle  = \psi_{N,k}(\mu_j^{(N)}, \omega_j^{(N)})
$$
as needed.

\item The proof is identical to the proof of Lemma~\ref{LEM:OptimizationProblemSolution}.

\item
\begin{align*}
\langle \psi_{N,j}, \phi_{N,k}\rangle
&=
\left\langle K_{2N}(\cdot, \cdot; \mu_j^{(N)}, \omega_j^{(N)})- K_{N}(\cdot, \cdot; \mu_j^{(N)}, \omega_j^{(N)}) , K_{N}(\cdot, \cdot; \rho_k^{(N)}, \theta_k^{(N)}) \right\rangle \\
&=
\left\langle \sum_{n = N+1}^{2N} \sum_{\substack{m=-n \\ (n-m) \textrm{even}}}^n \overline{Z^m_n(\mu_j^{(N)}, \omega_j^{(N)})} Z_n^m (r, \phi) , \sum_{n = 0}^{N} \sum_{\substack{m=-n \\ (n-m) \textrm{even}}}^n \overline{Z^m_n(\rho_k^{(N)}, \theta_k^{(N)})} Z_n^m (r, \phi) \right\rangle \\
&=
0
\end{align*}
due to orthogonality of the Zernike polynomials.

\end{enumerate}
\end{proof}

\paragraph{\textbf{Linear independence of the wavelet functions and the QR algorithm}}
Considering our goal of getting bases for the $W_N$s, we note that one cannot expect $\{\psi_{N,j}(r, \phi)\}_{j=1}^{D_N}$ to be linearly independent for an arbitrary set of points $\{(\mu^{(N)}_{j}, \omega_j^{(N)})\}_{j=1}^{D_N}.$ Selecting points that will guarantee linear independence of the set $\{\psi_{N,j}(r, \phi)\}_{j=1}^{D_N},$ along with good conditioning, poses an important and challenging problem.  Recently, the authors were made aware of the \textit{iterated QR algorithm} that may be used for this purpose \cite{Bos2008, Sommariva2009}. We outline the main idea here in our setting.   Start with a sufficiently large and dense discretization of the unit circle: $\{\rho_i, \theta_i\} \subset B(0,1).$ For the sake of explanation, take $\{\rho_i, \theta_i\}_{i=1}^{J_{2N}}$ as the set of regular points (Definition~\ref{DEF:SpecialPoints}) used to construct scaling functions for $V_{2N},$ but if needed, a different set of more densely spaced points can be taken. Then, let $A_N$ be the rectangular  \textit{Vandermonde} matrix 
	$$
	A_N = [Z_{ j}(\rho_i, \theta_i)]_{\substack{1\leq i \leq J_{2N} \\J_N < j \leq J_{2N}}}.
	$$
A greedy algorithm, based on the QR factorization, has been shown \cite{Bos2008, Sommariva2009} to extract a nonsingular Vandermonde submatrix of $A_N.$ The resulting $D_N$ points that are selected, called ``approximate Fekete points", aim to ``maximize'' the Vandermonde determinant absolute value and are thus also ``good" interpolation points. Hence, such a set of points $\{ \mu_j^{(N)}, \omega_j^{(N)} \}_{j=1}^{D_N}$ would be an appropriate candidate for constructing a basis of wavelets functions for $W_N.$

For our experiments, we have made a random selection of $D_N$ points from the set of regular points given in Definition~\ref{DEF:SpecialPoints} by setting $N = 2N$ there. That is, we take the points needed for scaling functions of $V_{2N}$ and take a random subset of $D_N$ points for the wavelet functions in $W_N.$
Figure~\ref{FIG:SamplingPoints} shows points for $N=3$ taken according to Definition~\ref{DEF:SpecialPoints} for the scaling functions. For the wavelets, we take the set of sampling points needed for scaling functions in $V_{2N},$ i.e., $V_6,$ and then take a subset of $D_3 = 18$ randomly selected points  from that set. Both are shown in Figure~\ref{FIG:SamplingPoints} . A few wavelet functions constructed by (\ref{EQ:WaveletZCP}) are shown in Figure~\ref{FIG:Wavelets}. The values of $r$ and $t$ shown in Figure~\ref{FIG:Wavelets} represent the values of $\rho$ and $\theta$, respectively, corresponding to each wavelet function, with $t$ being in radians.

%
%
\begin{figure}[!h]
\centering
\includegraphics[width=.4 \textwidth]{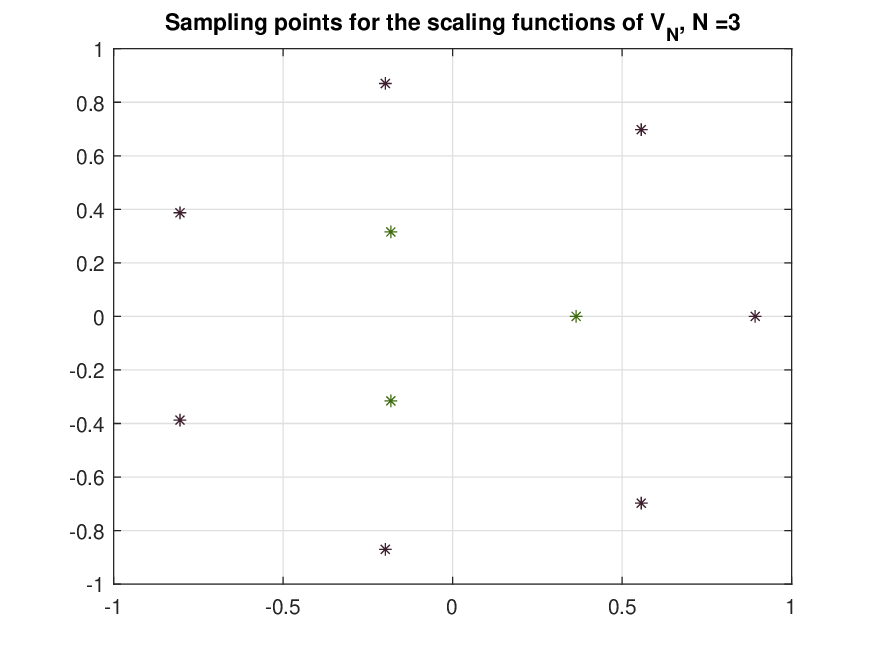}
\includegraphics[width=.4\textwidth]{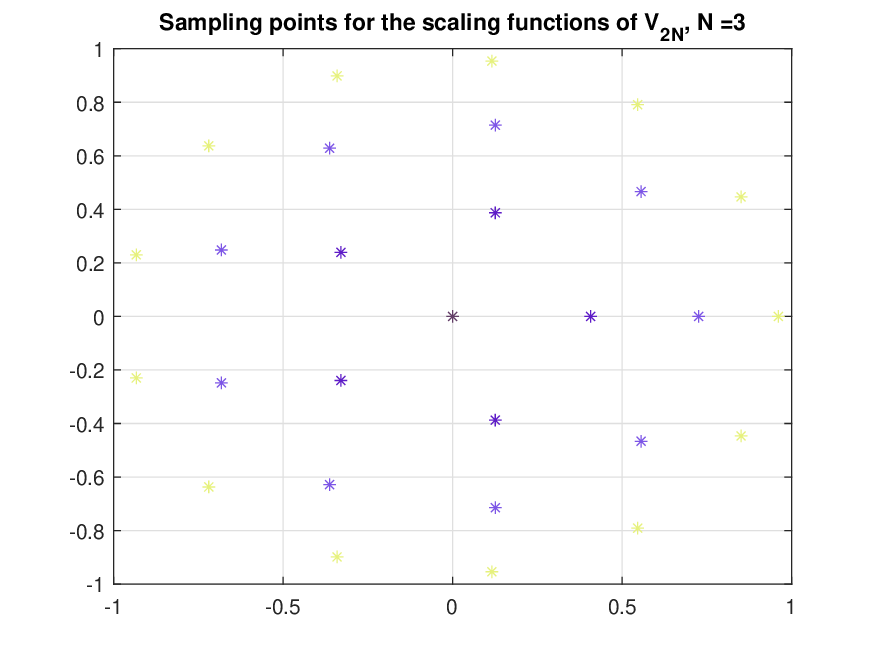}
\includegraphics[width=.4 \textwidth]{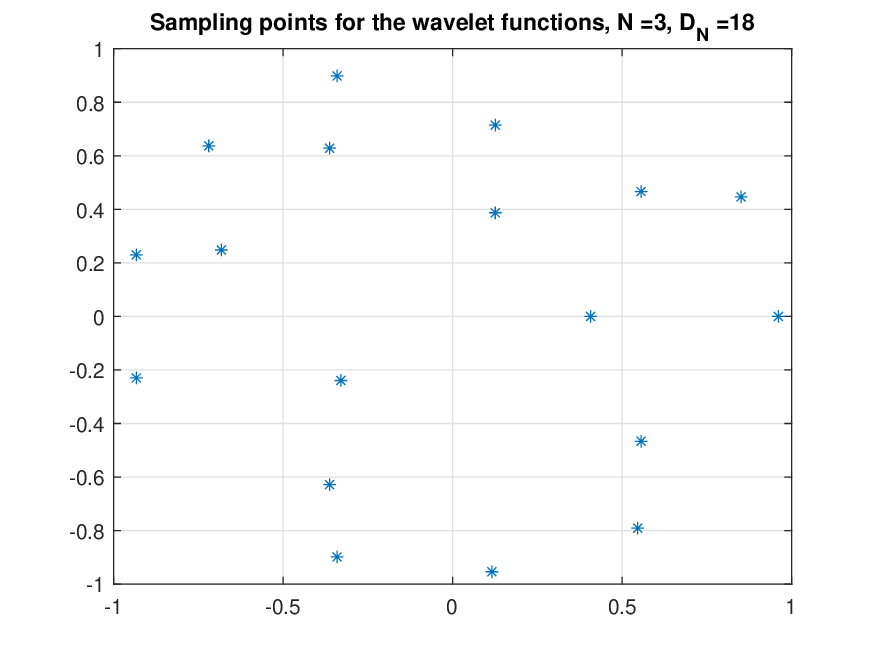}
\caption{Sampling points;  $N = 3$.}
\label{FIG:SamplingPoints}
\end{figure}
\begin{figure}[!h]
\centering
\includegraphics[width=.4 \textwidth]{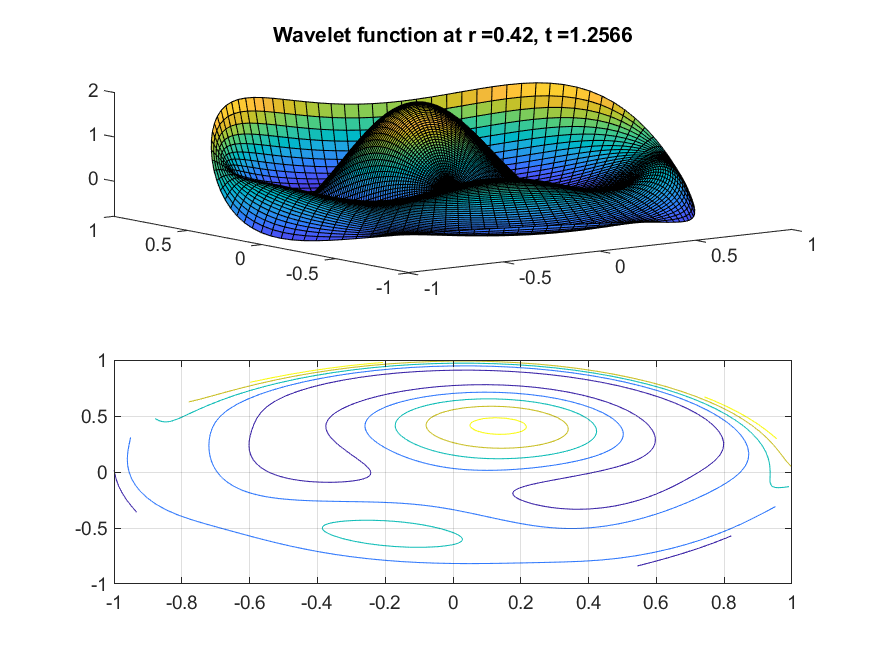}
\includegraphics[width=.4\textwidth]{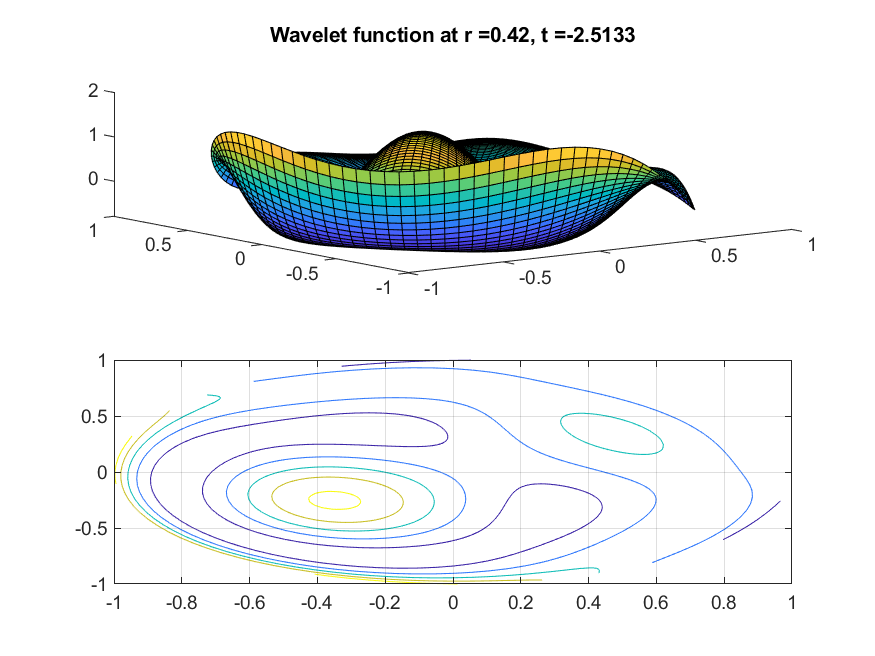}
\includegraphics[width=.4 \textwidth]{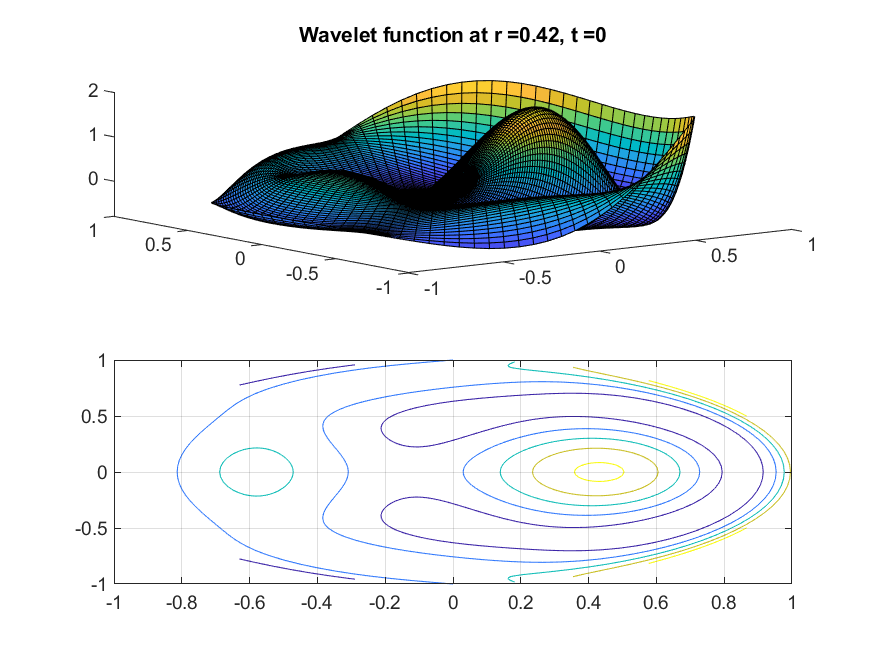}
\caption{Surfaces and contour lines of some wavelet functions, $N = 3$.}
\label{FIG:Wavelets}
\end{figure}

\section{Dual functions}\label{SEC:Duals}

Even when one has a set of points $\{(\mu^{(N)}_{j}, \omega_j^{(N)})\}_{j=1}^{D_N}$ (say by using a QR algorithm as discussed in Section~\ref{SEC:Wavelets}) so that the wavelet functions constructed from these points is a basis for $W_N$, it may still be impossible to get orthogonal wavelets. In that case, in order to reconstruct any signal, one needs to find the dual of the wavelet basis. We illustrate a method of finding a matrix representation for the dual. This approach may be useful in computations. 

Consider a wavelet basis $\{\psi_{N, j}(r, \varphi)\}_{j=1}^{D_N}$ of $W_N.$ Take $f \in W_N.$ This can be written as
	$$
	 f = \sum_{j = 1}^{D_N} c_{j} \psi_{N, j} .
	$$
The scalars $c_j$ are referred to as the \textit{wavelet coefficients} of $f.$ This is encountered in the numerical examples in Section~\ref{SEC:NumResults}.
Denote the dual of the wavelet basis by $\{\widetilde{\psi}_{N, j}(r, \phi)\}_{j=1}^{D_N}.$ Then $f$ can be written as
    $$
    f = \sum_{j = 1}^{D_N}\langle f, \psi_{N, j} \rangle \widetilde{\psi}_{N, j} .
    $$
Due to Corollary~\ref{COR:ONB_PI_N}, we know that $\{Z_{\ell}(r, \phi)\}_{\ell = J_N + 1}^{J_{2N}}$ is an ONB for $W_N = V_{2N} \ominus V_N.$ In terms of this ONB, the function $f$ can be written as
    $$
    f = \sum_{\ell = J_N + 1}^{J_{2N}} f^{(\ell)} Z_{\ell}(r, \phi).
    $$
Let us identify $f$ with the vector $[f^{(J_N + 1)}  \ \ldots \ f^{(J_{2N})}]^{\textrm{\tiny{T}}}.$
Recall that the $j$th wavelet function is
    \begin{equation} \label{EQ:Wavelet}
    \psi_{N, j}(r, \phi) = \sum_{\ell = J_N + 1}^{J_{2N}} \overline{Z_{\ell}(\mu_j, \omega_j)} Z_{\ell}(r, \phi),
    \end{equation}
where the superscript $N$ has been dropped from the parameter points for convenience of notation. We identify $\psi_{N, j}(r, \phi)$ with the vector
$[Z_{J_N + 1}(\mu_j, \omega_j) \ Z_{J_N + 2}(\mu_j, \omega_j) \ \ldots \ Z_{J_{2N}}(\mu_j, \omega_j)]^*,$ where $*$ represents the complex conjugate transpose. Consider the matrix $A_N$ whose complex conjugate transpose is
    $$
    A_{N}^*
    =
    \begin{bmatrix}
    \overline{Z_{J_N+1}}(\mu_1, \omega_1) & \overline{Z_{J_N+1}}(\mu_2, \omega_2) & \cdots & \overline{Z_{J_N+1}}(\mu_{D_N}, \omega_{D_N}) \\
    \overline{Z_{J_N+2}}(\mu_1, \omega_1) & \overline{Z_{J_N+2}}(\mu_2, \omega_2) &\cdots & \overline{Z_{J_N+2}}(\mu_{D_N}, \omega_{D_N}) \\
    \vdots & \vdots & \vdots & \vdots \\
    \overline{Z_{J_{2N}}}(\mu_1, \omega_1) & \overline{Z_{J_{2N}}}(\mu_2, \omega_2) &\cdots & \overline{Z_{J_{2N}}}(\mu_{D_N}, \omega_{D_N})
    \end{bmatrix}.
    $$
 Note that the $j$th column of $A_N^*$ represents the $j$th wavelet function $\psi_{N, j}.$ In vector notation, it can be seen that
    $$
    A_N f
    =
    \begin{bmatrix}
    \langle f, \psi_{N,1}\rangle \\
    \langle f, \psi_{N,2}\rangle \\
    \vdots \\
    \langle f, \psi_{N,D_N}\rangle \\
    \end{bmatrix} .
    $$
 Therefore,
    \begin{align*}
    A_N^* A_N f
    &=
   \begin{bmatrix}
    \overline{Z_{J_N+1}}(\mu_1, \omega_1) & \overline{Z_{J_N+1}}(\mu_2, \omega_2) & \cdots & \overline{Z_{J_N+1}}(\mu_{D_N}, \omega_{D_N}) \\
    \overline{Z_{J_N+2}}(\mu_1, \omega_1) & \overline{Z_{J_N+2}}(\mu_2, \omega_2) &\cdots & \overline{Z_{J_N+2}}(\mu_{D_N}, \omega_{D_N}) \\
    \vdots & \vdots & \vdots & \vdots \\
    \overline{Z_{J_{2N}}}(\mu_1, \omega_1) & \overline{Z_{J_{2N}}}(\mu_2, \omega_2) &\cdots & \overline{Z_{J_{2N}}}(\mu_{D_N}, \omega_{D_N})
    \end{bmatrix}
    \begin{bmatrix}
    \langle f, \psi_{N,1}\rangle \\
    \langle f, \psi_{N,2}\rangle \\
    \vdots \\
    \langle f, \psi_{N,D_N}\rangle \\
    \end{bmatrix} \\
    &=
    \langle f, \psi_{N,1} \rangle
    \begin{bmatrix}
    \overline{Z_{J_N+1}}(\mu_1, \omega_1) \\
    \overline{Z_{J_N+2}}(\mu_1, \omega_1) \\
    \vdots \\
    \overline{Z_{J_{2N}}}(\mu_1, \omega_1)
    \end{bmatrix}
    +
    \cdots
    +
    \langle f, \psi_{N,D_N} \rangle
    \begin{bmatrix}
    \overline{Z_{J_N+1}}(\mu_{D_N}, \omega_{D_N}) \\
    \overline{Z_{J_N+2}}(\mu_{D_N}, \omega_{D_N}) \\
    \vdots \\
    \overline{Z_{J_{2N}}}(\mu_{D_N}, \omega_{D_N})
    \end{bmatrix}
    \\
    &= \sum_{j=1}^{D_N} \langle f, \psi_{N,j} \rangle \psi_{N,j}
    \\
    \implies f &= \sum_{j=1}^{D_N} \langle f, \psi_{N,j} \rangle (A_N^* A_N)^{-1} \psi_{N,j}  = \sum_{j=1}^{D_N} \langle f, (A_N^* A_N)^{-1}\psi_{N,j} \rangle  \psi_{N,j}\\
    \implies \widetilde{\psi}_{N,j} &= (A_N^* A_N)^{-1} \psi_{N,j}.
    \end{align*}
%
The wavelet functions $\{\psi_{N,r}(r, \phi)\}_{r = 1}^{D_N}$ form an orthonormal set if and only if $A_N^* A_N$ is a diagonal matrix. Since this may not be possible for the choice of points used to construct the wavelets, one has to compute $(A_N^* A_N)^{-1}$ and then the dual $\{\widetilde{\psi}_{N,j}(r, \phi)\}_{j = 1}^{D_N}$ to reconstruct $f$ from the wavelet coefficients $\{\langle f, \psi_{N,j} \rangle\}_{j=1}^{D_N}.$ 

Another way to think about dual functions of wavelets in matrix or vector form is the following. In the expression given in (\ref{EQ:Wavelet}), let us consider sampling each function $\psi_{N,j}$ at $D$ points. This in turn implies that we are sampling each Zernike polynomial $Z_{\ell}(r, \phi)$ at $D$ points so that each Zernike polynomial $Z_{\ell}(r, \phi)$ can be thought of as a vector of $D$ elements. Equation (\ref{EQ:Wavelet}) can then be written as 
	\begin{eqnarray*}
	\psi_{N,j}(r, \phi) &=& \overline{Z_{J_N + 1}}(\mu_j, \omega_j) Z_{J_N + 1} (r, \phi) + \cdots + \overline{Z_{J_{2N}}}(\mu_j, \omega_j) Z_{J_{2N}} (r, \phi) \\
	&=& 
	\begin{bmatrix}
	| & | &  & | \\
	Z_{J_N + 1} & Z_{J_N + 2} & \cdots & Z_{J_{2N}} \\
	| & | &  & |
	\end{bmatrix} 
	\begin{bmatrix}
	 \overline{Z_{J_N + 1}}(\mu_j, \omega_j) \\
	 \vdots \\
	  \overline{Z_{J_{2N}}}(\mu_j, \omega_j)
	\end{bmatrix} \\
	&=& B \begin{bmatrix}
	 \overline{Z_{J_N + 1}}(\mu_j, \omega_j) \\
	 \vdots \\
	  \overline{Z_{J_{2N}}}(\mu_j, \omega_j)
	\end{bmatrix} .
	\end{eqnarray*} 
Note that $B$ is a $D \times D_N$ matrix. Let $\Psi$ be the matrix whose $j$th column is $\psi_{N,j}.$
$$
\Psi = \begin{bmatrix}
| & | &  & | \\
\psi_{N,1 } & \psi_{N,2} & \cdots & \psi_{N,D_N} \\
| & | & & |
\end{bmatrix}.
$$
Therefore, 
\begin{equation} \label{EQ:WaveletMatrix}
\Psi = B \begin{bmatrix}
\overline{Z_{J_N + 1}}(\mu_1, \omega_1) & \cdots & \overline{Z_{J_N + 1}}(\mu_{D_N}, \omega_{D_N}) \\
\vdots & \cdots & \vdots \\
\overline{Z_{J_{2N}}}(\mu_1, \omega_1) & \cdots & \overline{Z_{J_{2N}}}(\mu_{D_N}, \omega_{D_N})
\end{bmatrix} = B \ A_N^* .
\end{equation}
In the language of frame theory \cite{Chr03}, the $D \times D_N$ matrix $\Psi$ is the \textit{synthesis operator} of the discretized set of vectors $\{\psi_{N,j}\}_{j=1}^{D_N},$ and $\Psi \Psi^*$ is the \textit{frame operator} \cite{Chr03}. The dual wavelet functions, in discretized form, are then given by 
	$$
	\widetilde{\psi}_{N,j} = (\Psi \Psi^*)^{-1} \psi_{N,j} = (BA_N^* A_N B^*)^{-1} \psi_{N,j} . 
	$$

\section{Multiresolution Analysis}\label{SEC:MRA}

The scaling functions defined in Section~\ref{SEC:ScalingFunctions} give rise to a multiresolution analysis (MRA) of $L^2(B(0,1))$ that is different from the traditional MRA. Recall that $V_N$ is the subspace of all polynomials in $x,$ $y$ with degree at most $N$. This gives a one-sided MRA starting with $V_0$ that satisfies the following.

\begin{enumerate}[(i)]

\item $V_{2^j} \subset V_{2^{j+1}},$ $j \in \mathbb{N}_0,$
giving rise to the sequence of nested subspaces starting with $V_0$:
    $$
    V_0 \subset V_1 \subset \cdots \subset V_{2^j} \subset V_{2^{j+1}} \subset \cdots .
    $$

\item $V_0 \cap (\bigcap_{j \in \mathbb{N}_0} V_{2^j} ) = V_0 = \textrm{span}\{\mathds{1}_{B(0,1)} \},$ where $\mathds{1}_{B(0,1)}$ is the characteristic function of $B(0,1).$

\item $\overline{V_0 \cup (\bigcup_{j \in \mathbb{N}_0} V_{2^j}) } = L^2(B(0,1)).$

\item \textit{Dilation:} $p(x,y) \in V_{2^j} \Leftrightarrow p(x^2, y^2) \in V_{2^{j+1}}$
\\
$f \in V_n \Leftrightarrow \langle f, Z^m_k \rangle = 0$ for all $k > n.$

\item \textit{Translation:} $\{\phi_{N,r}\}_{r=1}^J$ based on a parameter set of regular  points (Definition~\ref{DEF:SpecialPoints}) is a basis of $V_N,$ where $N = 2^j$ for some $j \in \mathbb{N}_0$ and $J = \frac{(N+1)(N+2)}{2}.$
\end{enumerate}
\noindent
The last property, determined by the choice of the regular points, takes the role of integer translates in a standard MRA. For a given function $f$ and a given $N$ that is a power of $2$, one can obtain the best approximation of $f$ in the space $V_N$ by considering the orthogonal projection of $f$ on $V_N.$ Call this $f_{P_N}.$ We can write
    \begin{equation} \label{EQ:WaveletDecomposition}
    V_N = V_{\frac N2} \oplus W_{\frac N2} = V_{\frac N4} \oplus W_{\frac N4} \oplus W_{\frac N2} = \cdots = V_0 \oplus W_0 \oplus \cdots \oplus W_{\frac N2}.
    \end{equation}
Note that $V_0$ is spanned by constant polynomials, and we take wavelet functions that form a basis for the spaces $W_j.$ Thus (\ref{EQ:WaveletDecomposition}) implies that 
one can write $f_{P_N}$ as a linear combination of the wavelets. Taking higher values of $N$ will make $f_{P_N}$ a better approximation of $f.$

At this point, we recollect the advantages of the wavelet representation over the Zernike basis representation ((\ref{EQ:ZernikeFourierExp}), Corollary~\ref{COR:ONB_PI_N}). Note that the degrees of the polynomials represent the frequencies; see also property (iv) above. When $f$ is represented in terms of the Zernike polynomials, a truncation of the series (\ref{EQ:ZernikeFourierExp}) may give a good enough approximation of $f,$ depending on the number of terms taken; however, one does not have any knowledge of the location of the frequencies corresponding to the coefficients. For a fixed $N,$ the scaling functions give \textit{localized} bases of $V_N,$ but the scaling functions being kernel polynomials, are composed of \textit{all} Zernike polynomials upto a certain degree and hence approximating $f$ by a linear combination of scaling functions does not give enough information on the frequencies present. However, as seen in (\ref{EQ:WaveletDecomposition}), the approximation of $f$ in $V_N,$ when expressed as a sum of components in the $W_j$s, will give information on the various frequencies present as well as the location of the frequencies. More precisely, in the wavelet representation, the degrees of the polynomials contained in each $W_j$ gives the frequency information, whereas the location of the frequencies comes from the parameter points used to construct the wavelet functions.

\section{Numerical results and discussion} \label{SEC:NumResults}
In this section we show some experimental results that demonstrate the use of wavelets constructed using Zernike polynomials. The data used comes from normal subjects, those with corneal astigmatism, and those with keratoconus. A Medmont E300 videokeratoscope was used to obtain the data.\footnote{We thank Dr. D. Robert Iskander for kindly providing the data.} 
The data points for each subject are stored as a vector $C$ of size $D.$ For the data used, $D = 10200.$
This data, taken from the right eye of a subject with astigmatism, is shown on the top left of Figure~\ref{FIG:Approximation}. 
To represent this data as an approximation in some $V_N,$ in terms of Zernike polynomials, we first fix $N.$ For a given $N,$ there are $J = (N+1)(N+2)/2$ Zernike polynomials with radial degree at most 
$N$ that form a basis of $V_N.$ We can write \cite{Iskander2002}
    $$
    C = B a,
    $$
where $B$ is a $D \times J$ matrix whose columns are the $J$ Zernike polynomials (sampled at the $D$ points), and $a$ is a vector of coefficients. 
 Using the method of least-squares, the coefficient vector $a$ can be estimated as
    $$
    \widehat{a} = (B^{\textrm{\tiny{T}}} B)^{-1} B^{\textrm{\tiny{T}}} C.
    $$
The approximation of the data $C$ in the space $V_N$ is
    $$
    \widehat{C} = B \widehat{a}.
    $$
To start with, we take $N = 8,$ and note that any other power of $2$ can be taken; the higher the power, the better the approximation. The approximation of the (astigmatism) elevation data $C$ in the space $V_8$ is shown on the top right of Figure~\ref{FIG:Approximation}. Next, we do the same with $N$ replaced by $\frac N2,$ i.e., $4$ in this case. This gives an approximation of $C$ in the space $V_4.$ The difference between the two levels of approximation will belong to the space $W_{\frac N2}$ which is $W_4$. These are shown in the bottom row of Figure~\ref{FIG:Approximation}.

Using the above idea, for a given $N$ that is a power of $2$, one can obtain approximations in the spaces $V_N,$ $V_{\frac N2},$ $V_{\frac N4},$ and so on. We can write
    \begin{equation} \label{EQ:WaveletDecomposition2}
    V_N = V_{\frac N2} \oplus W_{\frac N2} = V_{\frac N4} \oplus W_{\frac N4} \oplus W_{\frac N2} = \cdots = V_0 \oplus W_0 \oplus \cdots \oplus W_{\frac N2}.
    \end{equation}
Note that $V_0$ is spanned by constant polynomials, and the choice of $N=8$ gives $J = 45$ which is the dimension of the subspace $V_8$ as well as the number of Zernike polynomials used to represent $V_8.$ This means that the matrix $B$ above is of size $10200 \times 45$. The dimension of the subspace $V_4$ is $15$ and so the wavelet space $W_4$ is of dimension $45 - 15 = 30.$ The next step is to represent the differences in the $W_j$s, i.e., the wavelet spaces, in terms of the wavelet basis functions mentioned in Section~\ref{SEC:Wavelets}. 
For the purpose of experimentation, for each $j$, the wavelet basis functions for $W_j$ are constructed using $D_j$ points that are a \textit{random} subset of  
the regular points used to construct scaling functions for $V_{2j}$ (Definition~\ref{DEF:SpecialPoints}). 
As indicated by (\ref{EQ:WaveletDecomposition2}), we can then represent our approximation in the space $V_N,$ for a given $N,$ in terms of the wavelet functions only. Including the constant function needed for $V_0,$ the number of wavelet functions needed for $N = 8$ would also be $45.$ The reconstruction of the elevation data of Figure~\ref{FIG:Approximation} using wavelet functions, the Zernike coefficients, and the wavelet coefficients are shown in Figure~\ref{FIG:WaveletCoefsAstig}. The bottom row in Figure~\ref{FIG:WaveletCoefsAstig} shows the wavelet coefficients (left) and the difference of the elevation data of Figure~\ref{FIG:Approximation} with the best-fit-sphere of that data (right). The difference with the best-fit-sphere (BFS) indicates the ridges and undulations in the data. By looking at the wavelet coefficients one can determine the locations where the data changes with higher frequencies. The norm of the difference between the elevation data and the approximation using Zernike polynomials as well as using wavelet functions is $0.07. $

Similar experiments have been performed with data from the right eyes of normal subjects and subjects with keratoconus using $N = 8$. An example of each is shown in Figures~\ref{FIG:Approximation2} -  \ref{FIG:WaveletCoefsKera}. The norm of the difference between the elevation data and the approximation using Zernike polynomials as well as using wavelet functions is $0.0682$ in the normal case and $0.1006$ for the subject with keratoconus.
One would expect better results giving more information when $N$ is larger. For $N = 16,$ the dimension of $V_{16}$ is $153$ which is the same as the number of wavelet functions to be used for this value of $N.$ See Figure~\ref{FIG:WaveletCoefsAstigN16} for the subject with astigmatism discussed earlier in  Figure~\ref{FIG:Approximation} and Figure~\ref{FIG:WaveletCoefsAstig}. 
Now there are many more wavelet coefficients and consequently one has a better knowledge of locations of fluctuations in spatial frequency. 
The norm of the difference between the elevation data and the approximation using Zernike polynomials as well as using wavelet functions is now $0.05.$ In comparing the wavelet coefficients obtained from $N = 8$ and $N = 16,$ it is worth recalling that the wavelet functions for $W_N$ are constructed from a random subset of points that are used to construct the scaling functions for $V_{2N}$. Each trial, with a given data set and a particular $N,$ uses a different set of wavelets functions, and this will also effect the location of the most significant wavelet coefficients. What is seen in Figures~\ref{FIG:WaveletCoefsAstig}, \ref{FIG:WaveletCoefsNormal}, \ref{FIG:WaveletCoefsKera}, and \ref{FIG:WaveletCoefsAstigN16} is just from a single trial. However, one can see that for the same astigmatism data with $N = 8$ and $N = 16,$  the one for $N=16$ (Figure~\ref{FIG:WaveletCoefsAstigN16}) has a lot more wavelet coefficients and hence gives much more information than the corresponding result for $N = 8$ (Figure~\ref{FIG:WaveletCoefsAstig}).
The difference between two different trials will be less noticeable as $N$ increases and more and more points are used. The results will be more meaningful with higher values of $N$ at the cost of more computational time. 
Table~\ref{TAB:NumericalResults} gives a summary of the results shown in the figures. It is seen that the approximation error when reconstructing using the Zernike basis is the same as when reconstructing using wavelet functions, at least for the cases that we have used. We would like to emphasize that the Fourier-Zernike series (\ref{EQ:ZernikeFourierExp}) only gives us the magnitude of the frequencies without any knowledge of where the frequencies occur. On the other hand, each wavelet coefficient can be associated with certain frequencies as well as a location coming from the parameter point used to construct the corresponding wavelet function. 
This information may be used to determine conditions such as astigmatism more efficiently by a wavelet analysis; the spatial-frequency behavior of a subject with some condition such as astigmatism may be different from that of a normal subject. However, actually detecting such conditions through wavelet analysis is beyond the scope of this paper, partly due to limited computer power at our disposal. We just aim to demonstrate the theory.
 
 Apart from the modeling of corneal elevation as illustrated here, there are a variety of other fields having circular boundary where the present theoretical results can be applied. For example, the wavelet analysis done here may also be used in feature extraction such as detection of circular shapes in images or textures in medical imaging \cite{FAROKHI2014, Kamal2016}.
This can also be used in processing radar signals, particularly in the context of circular radar beams, enabling more accurate target detection \cite{Cle2015}.
\begin{table}[ht]
\caption{Reconstruction of corneal data} 
\centering 
\begin{tabular}{c c c c c} 
\hline\hline 
Condition & Degree & No. of polynomials & Norm of approx. error & Norm of approx. error\\
& N & J & Zernike basis &  Wavelet functions \\ [0.5ex] 
\hline 
Astigmatism & 8 & 45 & 0.07 & 0.07 \\ 
Normal & 8 & 45 & 0.0682 & 0.0682 \\
Keratoconus & 8 & 45 & 0.1006 & 0.1006 \\
Astigmatism & 16 & 153 & 0.05 & 0.05  \\ [1ex] 
\hline 
\end{tabular}
\label{TAB:NumericalResults} 
\end{table}


\begin{figure}[!h]
\centering
\includegraphics[width=.45 \textwidth]{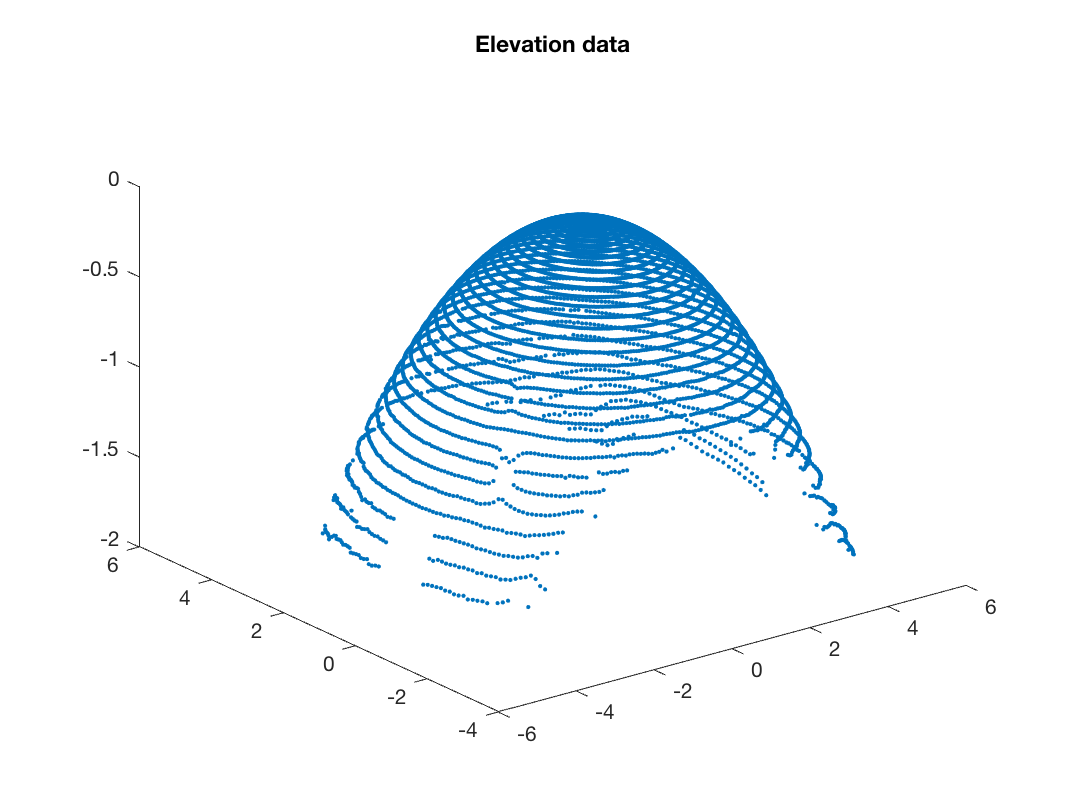}
\includegraphics[width=.45\textwidth]{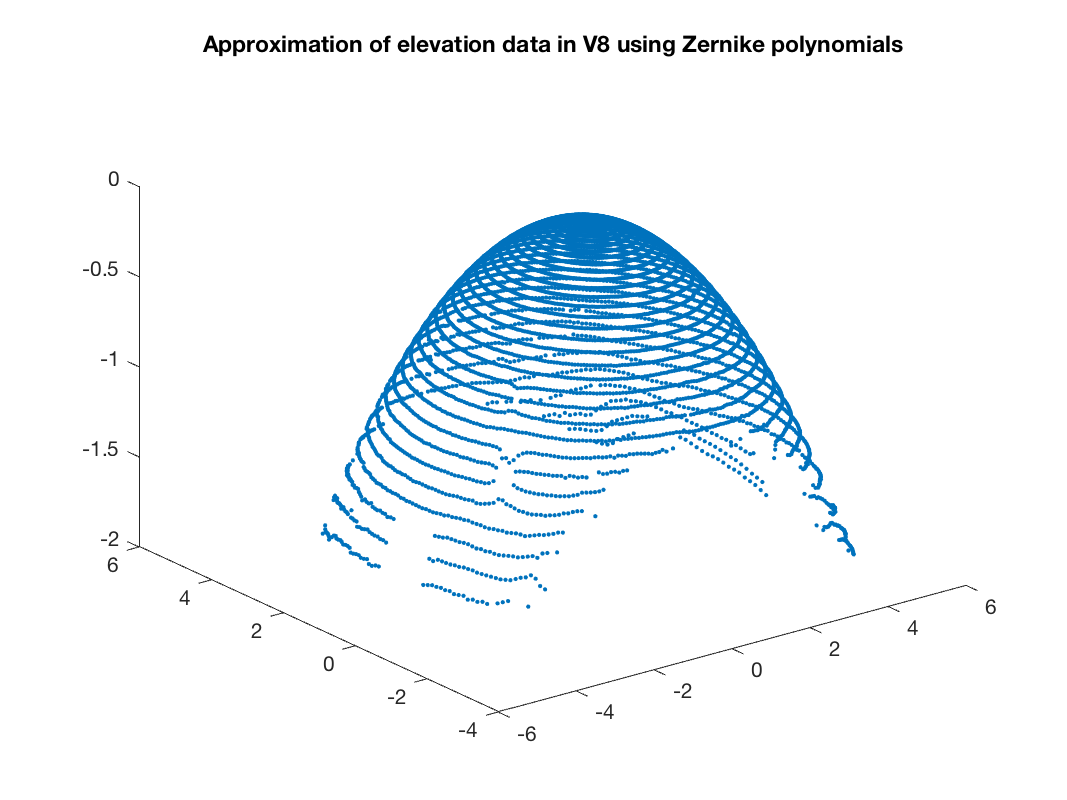}
\includegraphics[width=.45 \textwidth]{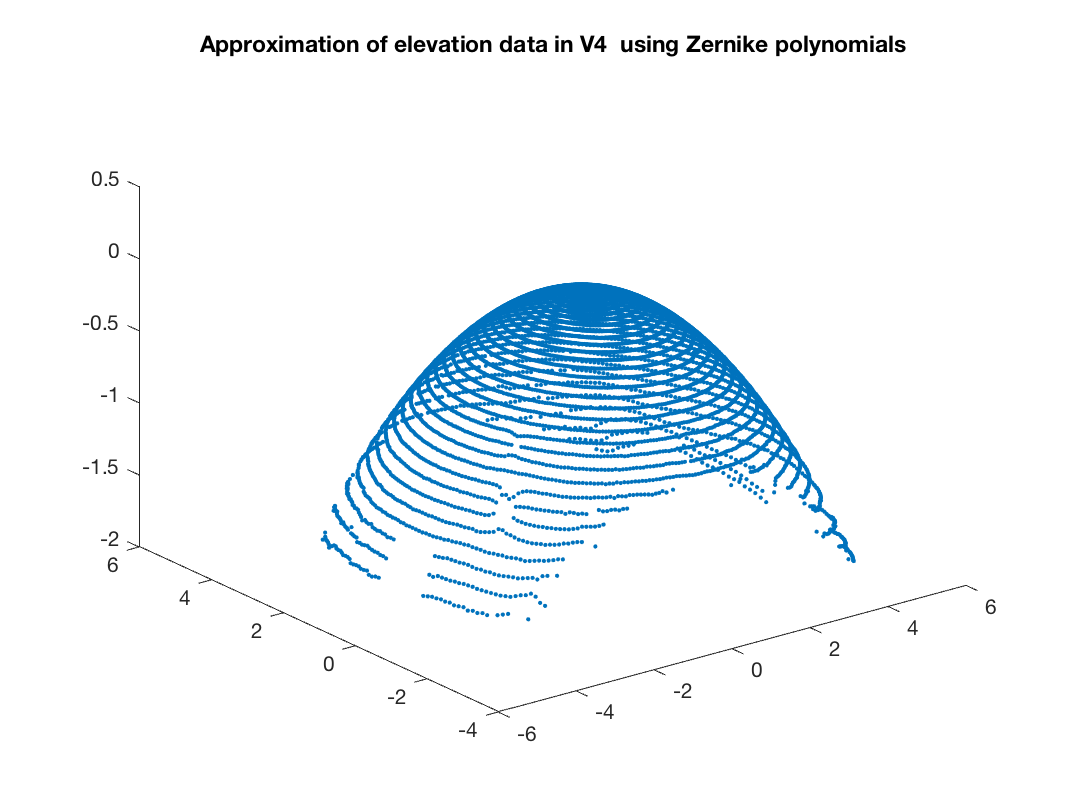}
\includegraphics[width=.45 \textwidth]{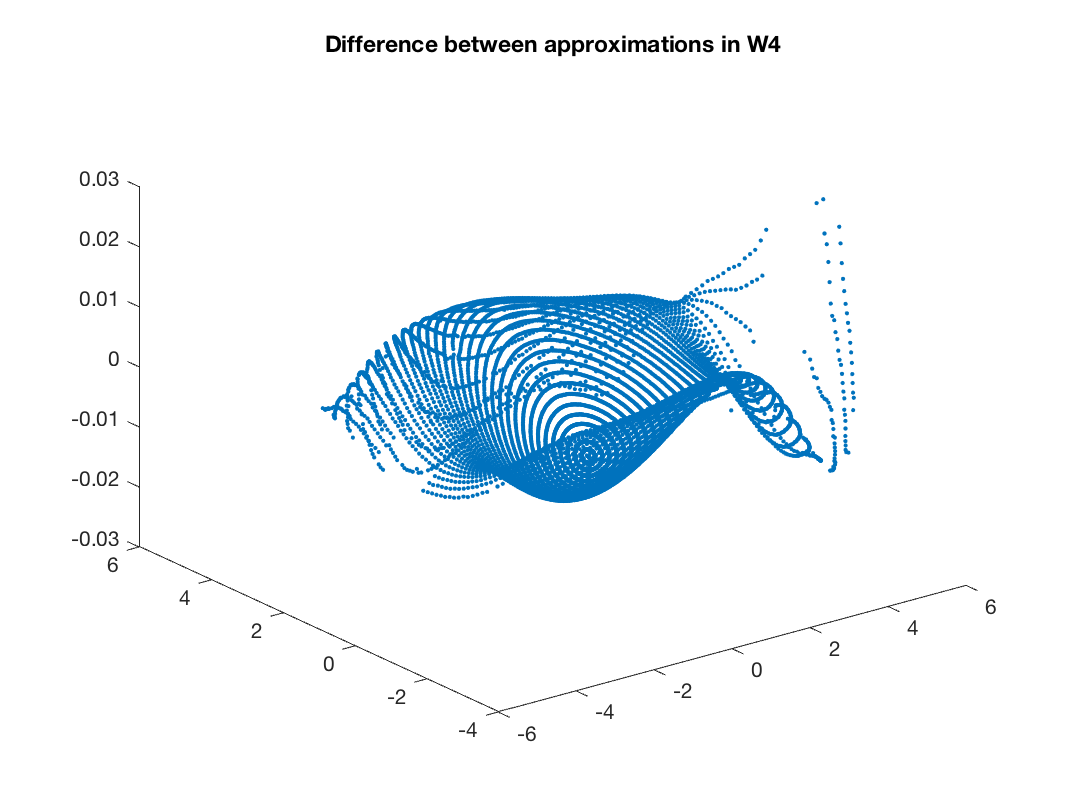}
\caption{Approximation of elevation data of a subject with astigmatism.}
\label{FIG:Approximation}
\end{figure}

\begin{figure}[!h]
\centering
\includegraphics[width=.45 \textwidth]{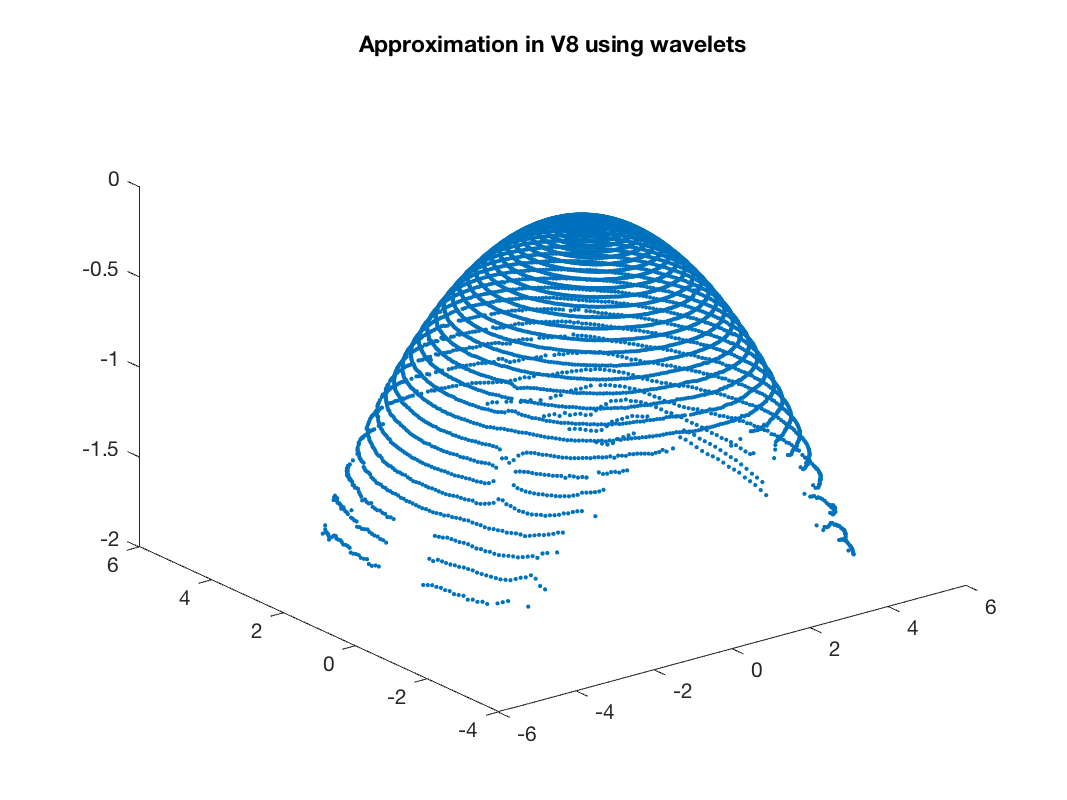}
\includegraphics[width=.45 \textwidth]{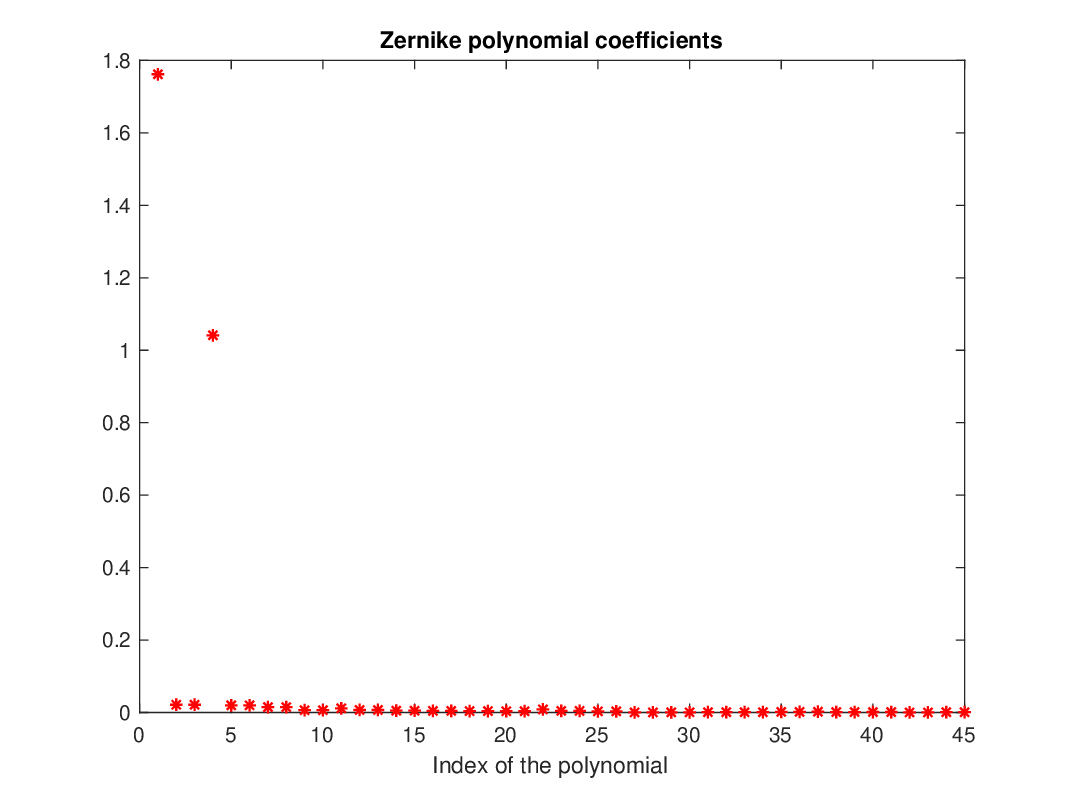}
\includegraphics[width=.45\textwidth]{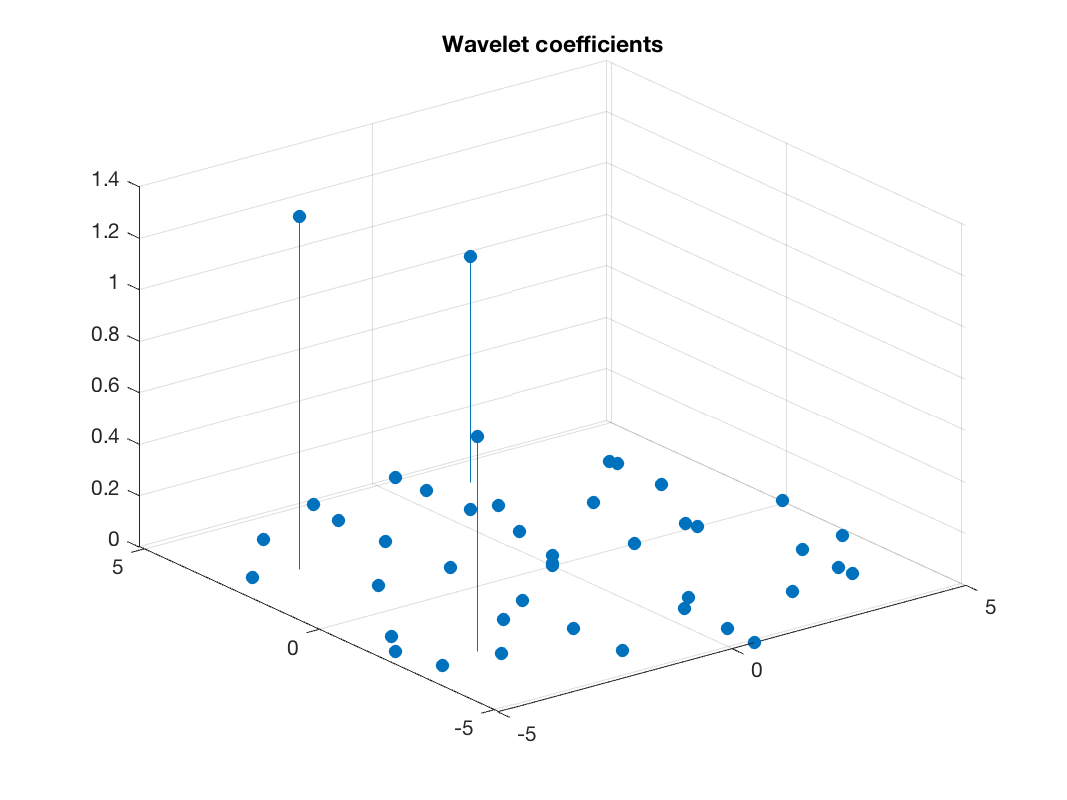}
\includegraphics[width=.45 \textwidth]{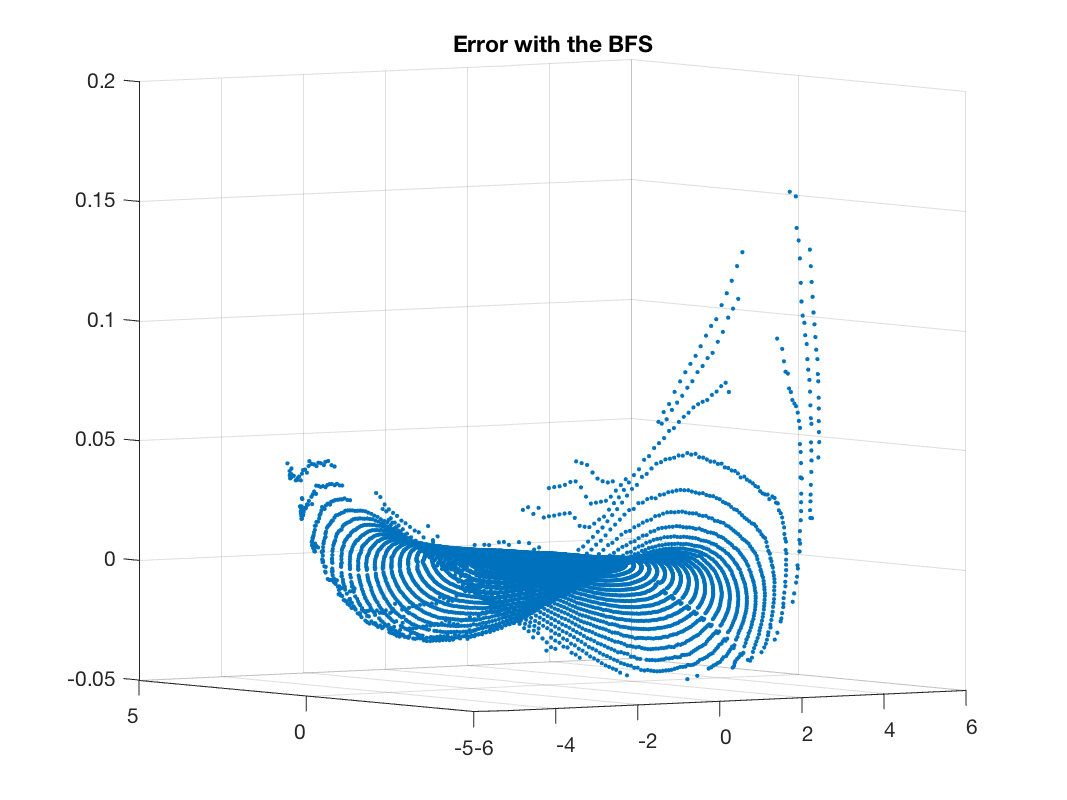}
\caption{Wavelet and Zernike coefficients of the astigmatism data of Figure~\ref{FIG:Approximation}}
\label{FIG:WaveletCoefsAstig}
\end{figure}

\begin{figure}[!h]
\centering
\includegraphics[width=.45 \textwidth]{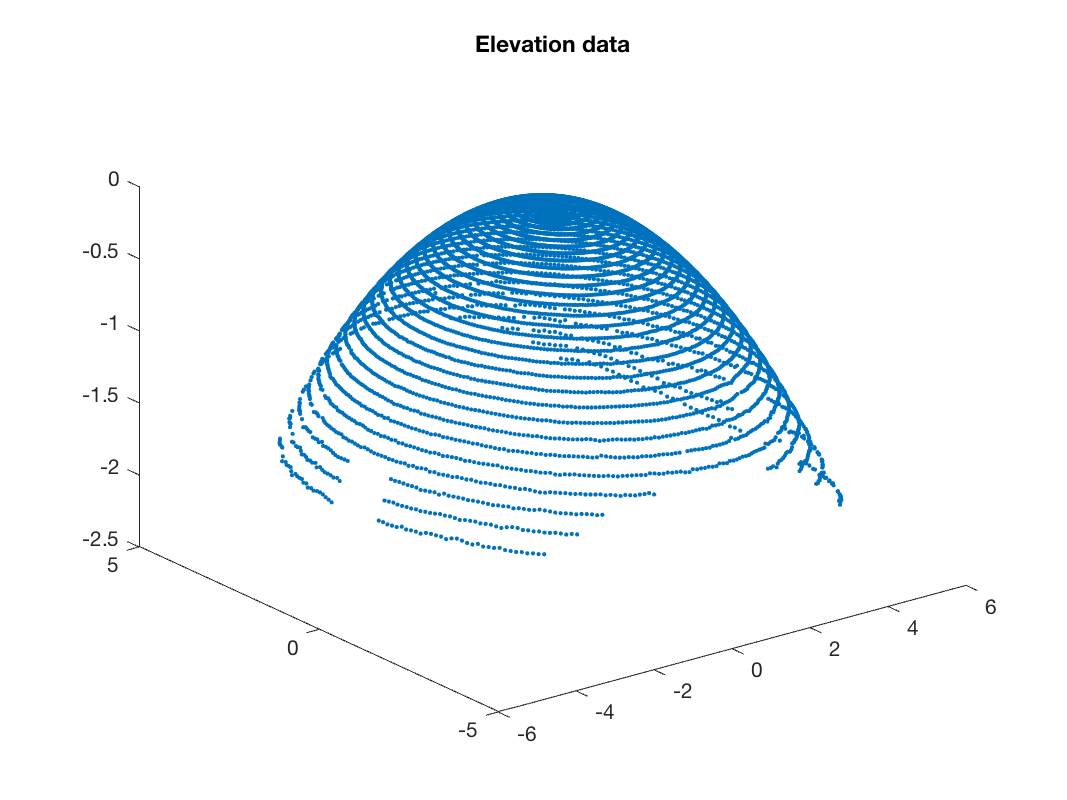}
\includegraphics[width=.45\textwidth]{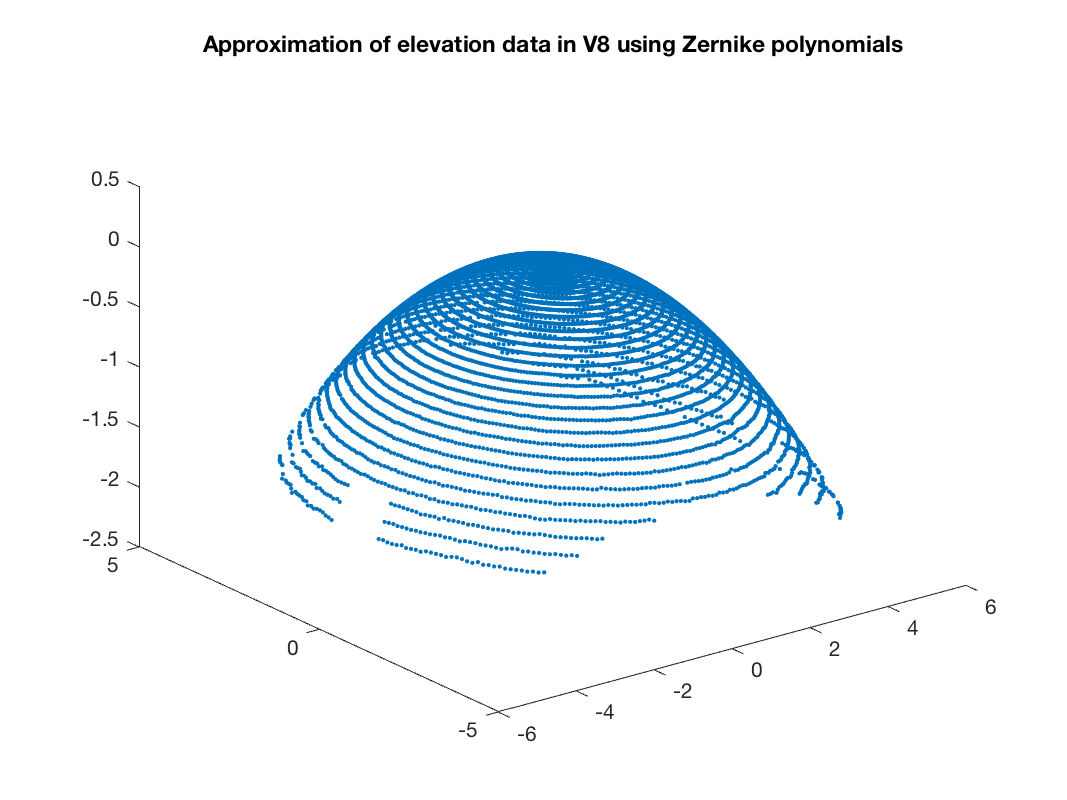}
\includegraphics[width=.45 \textwidth]{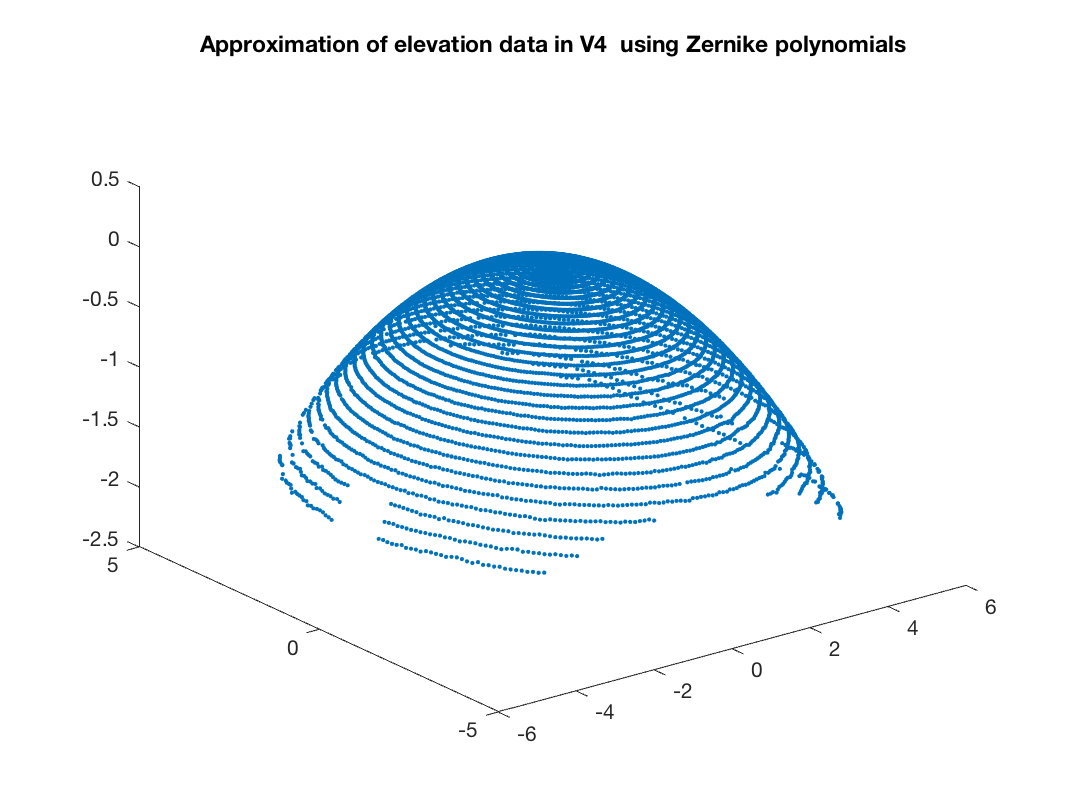}
\includegraphics[width=.45 \textwidth]{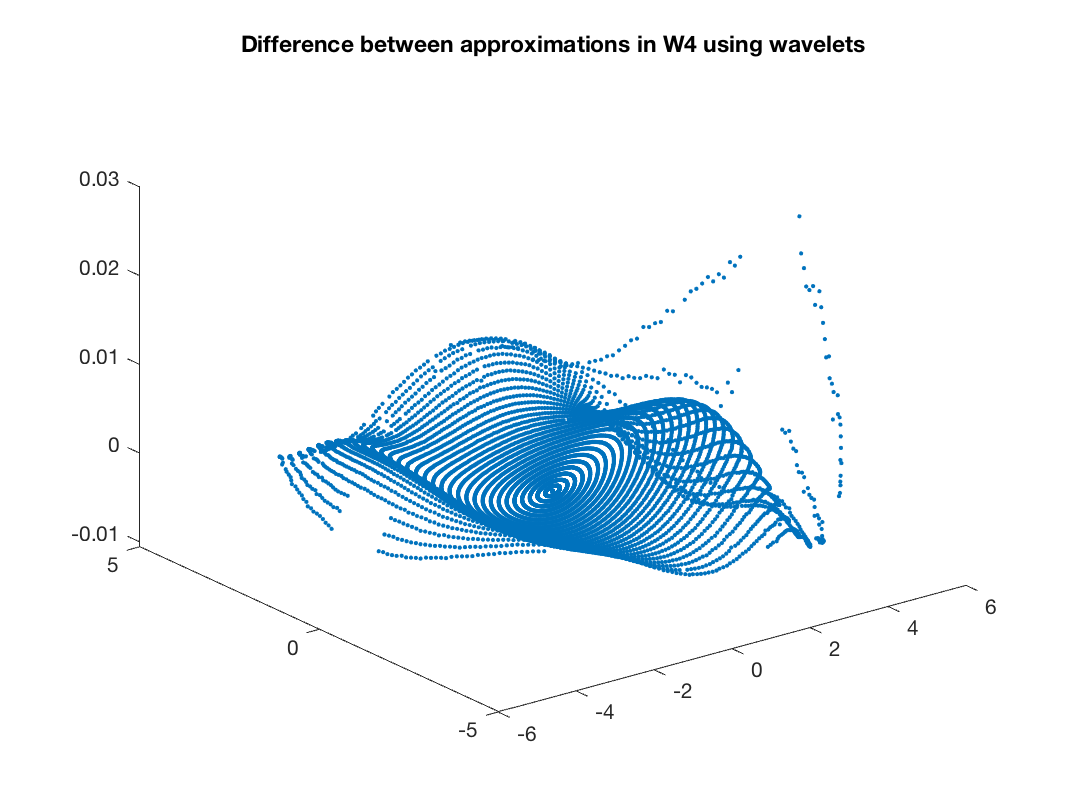}
\caption{Approximation of elevation data of a normal subject.}
\label{FIG:Approximation2}
\end{figure}

\begin{figure}[!h]
\centering
\includegraphics[width=.45 \textwidth]{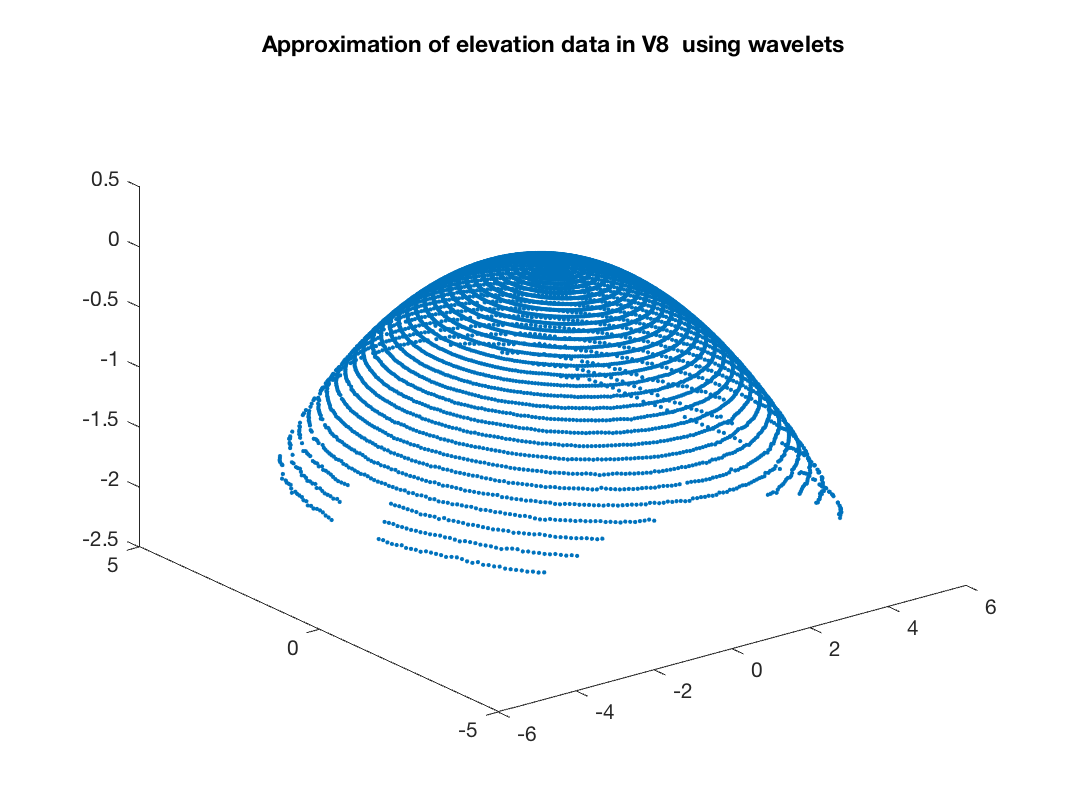}
\includegraphics[width=.45 \textwidth]{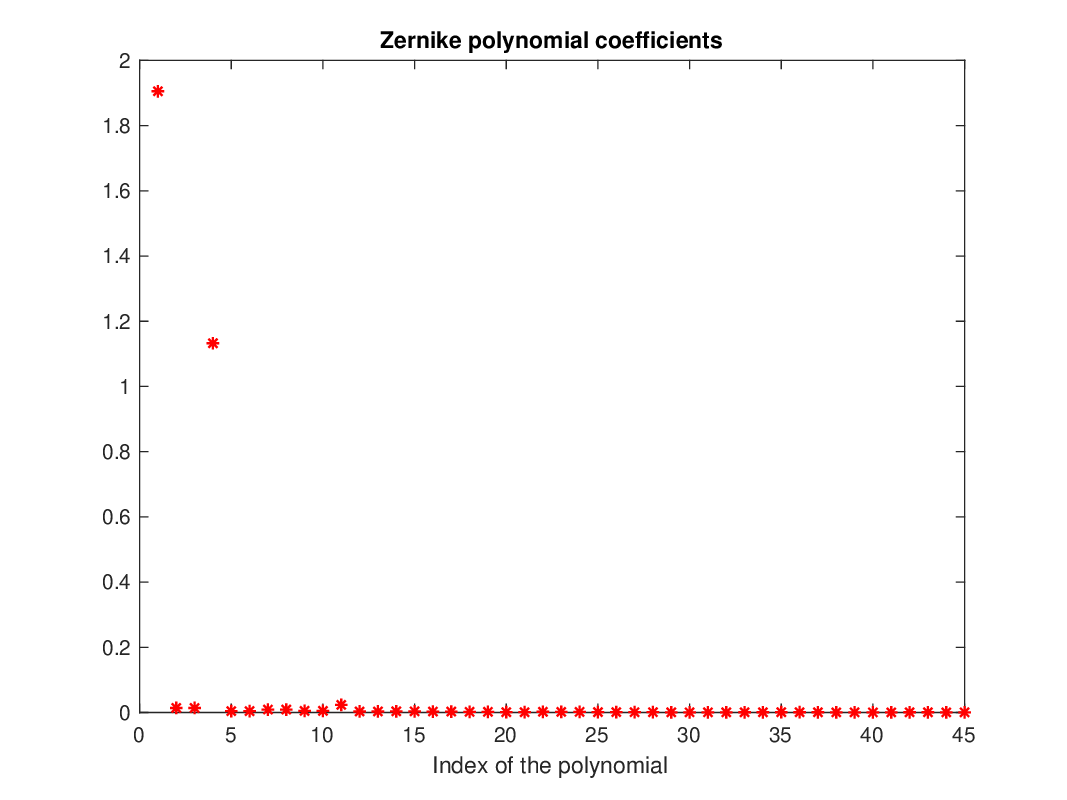}
\includegraphics[width=.45\textwidth]{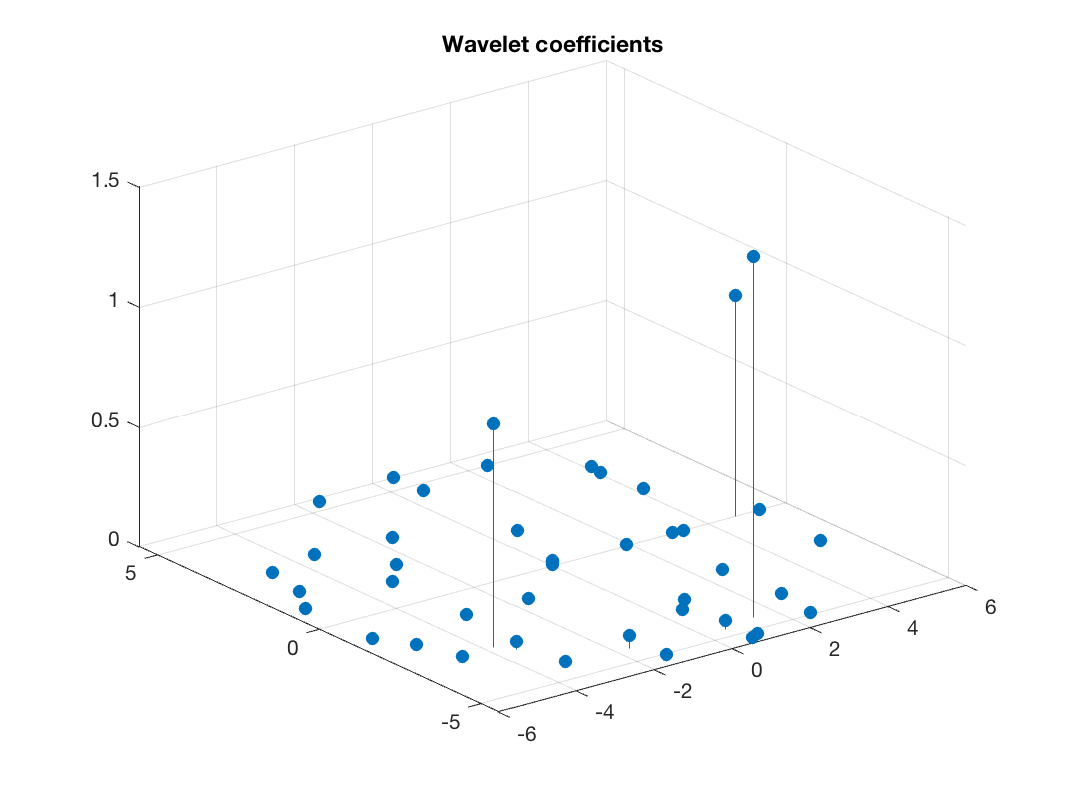}
\includegraphics[width=.45 \textwidth]{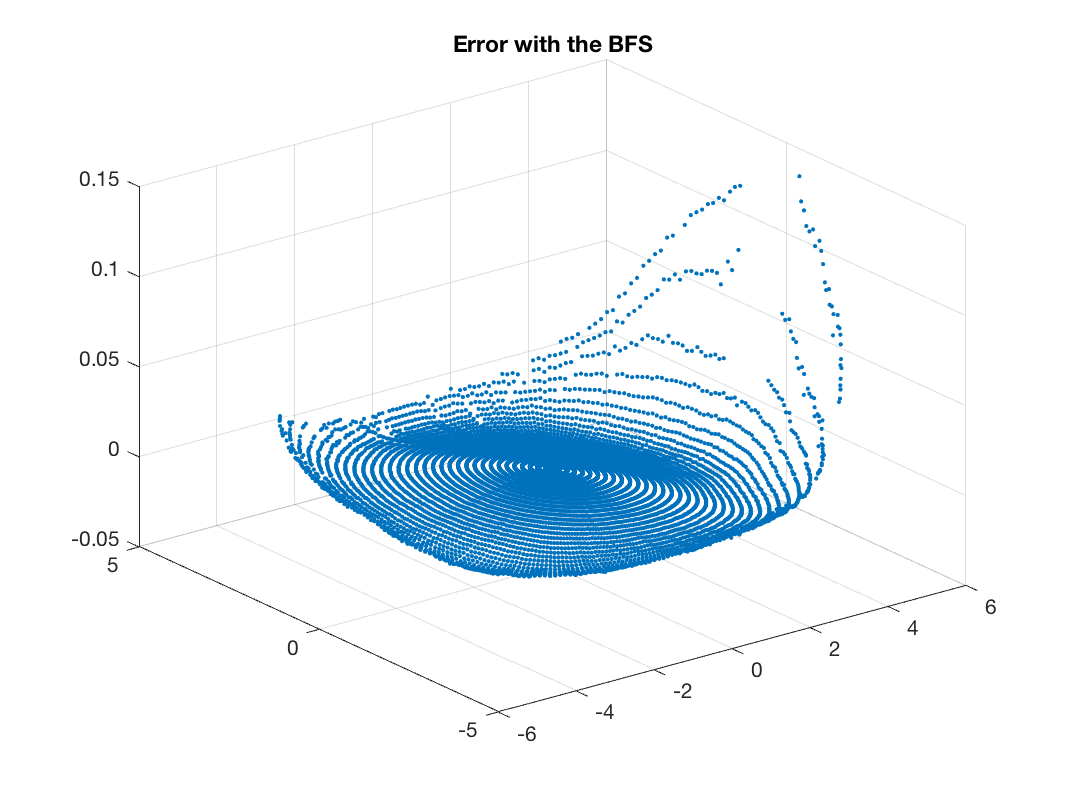}
\caption{Wavelet and Zernike coefficients of the data from a normal subject of Figure~\ref{FIG:Approximation2}}
\label{FIG:WaveletCoefsNormal}
\end{figure}

\begin{figure}[!h]
\centering
\includegraphics[width=.45 \textwidth]{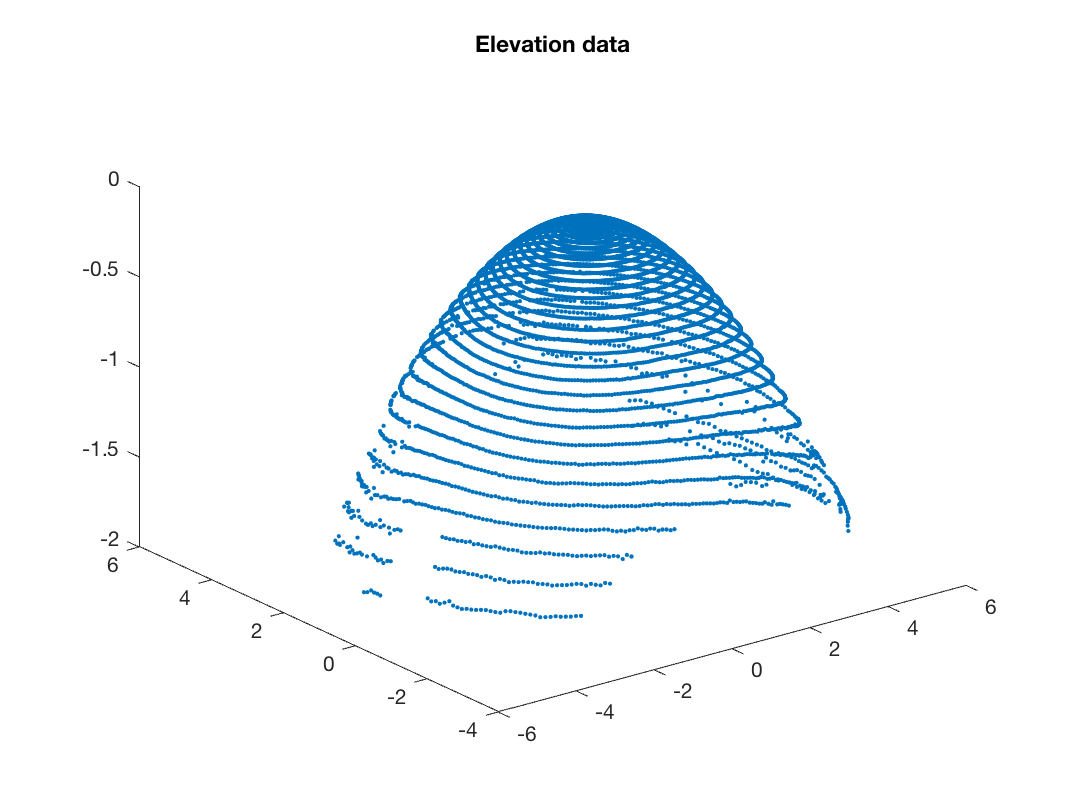}
\includegraphics[width=.45\textwidth]{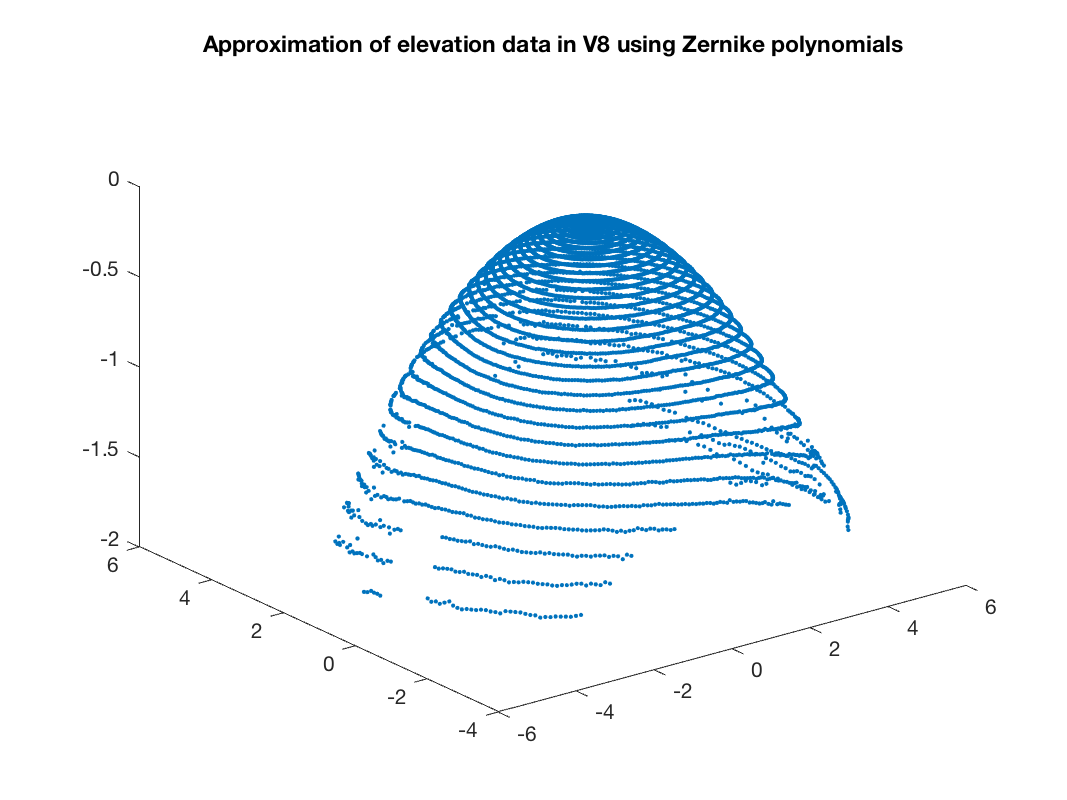}
\includegraphics[width=.45 \textwidth]{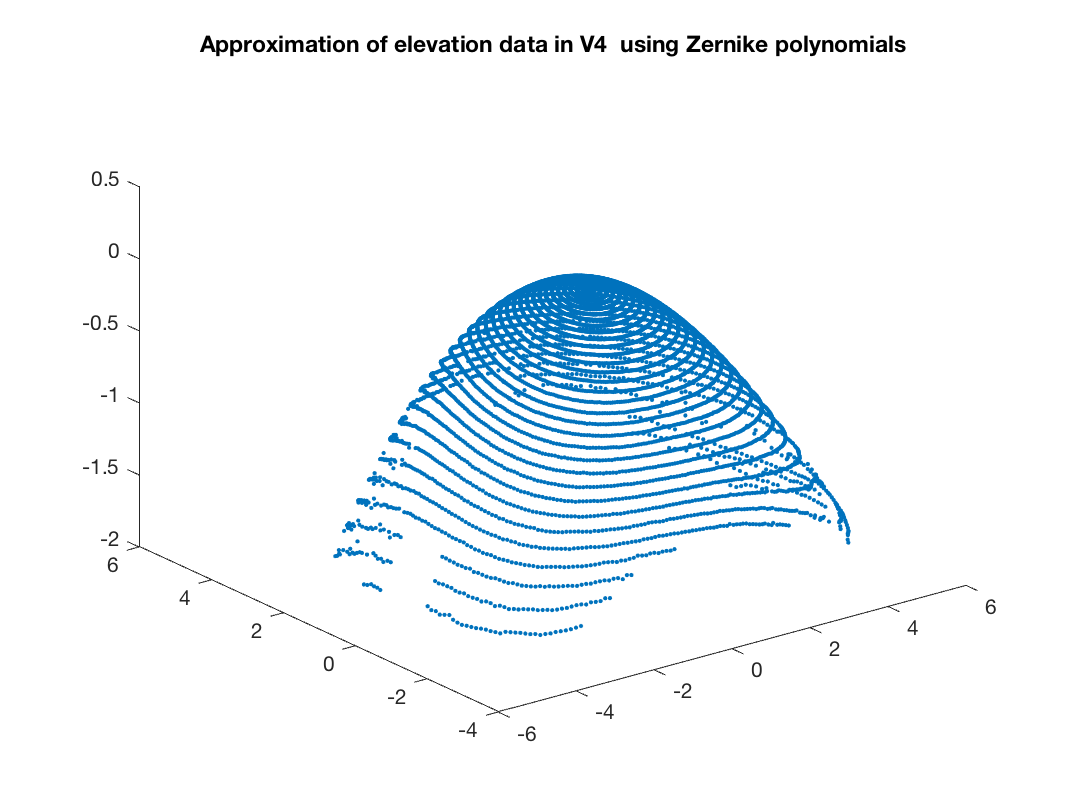}
\includegraphics[width=.45 \textwidth]{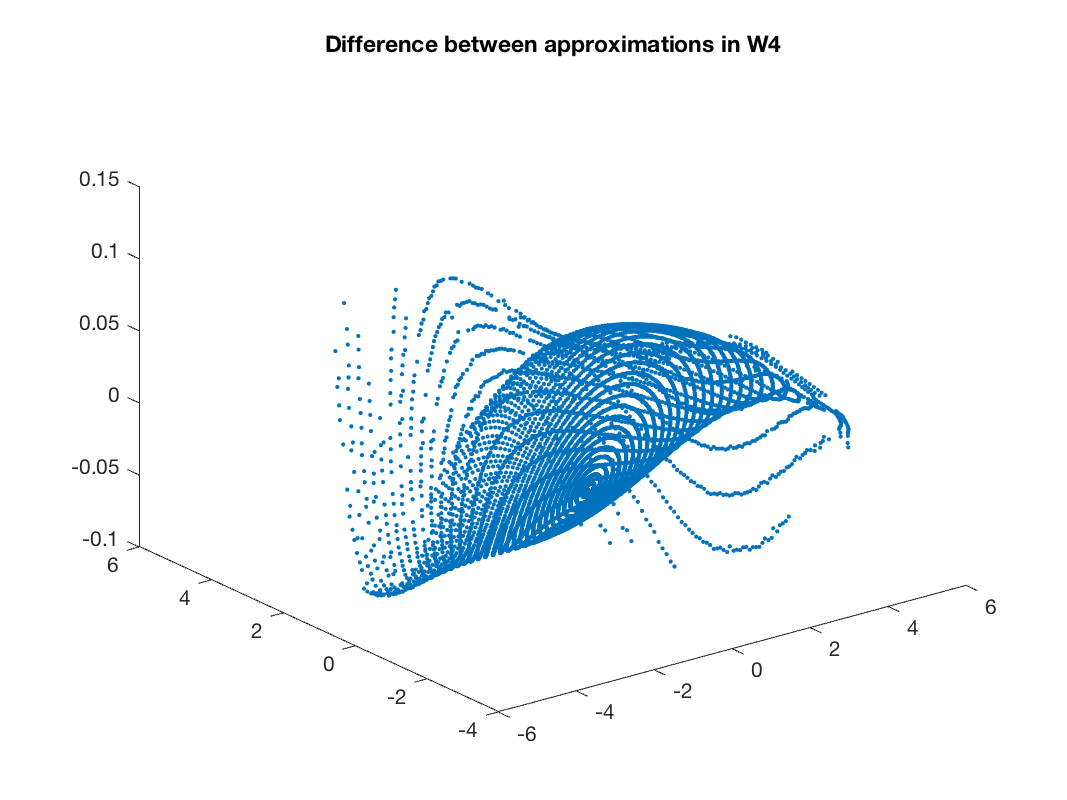}
\caption{Approximation of elevation data of a subject with keratoconus.}
\label{FIG:Approximation3}
\end{figure}

\begin{figure}[!h]
\centering
\includegraphics[width=.45 \textwidth]{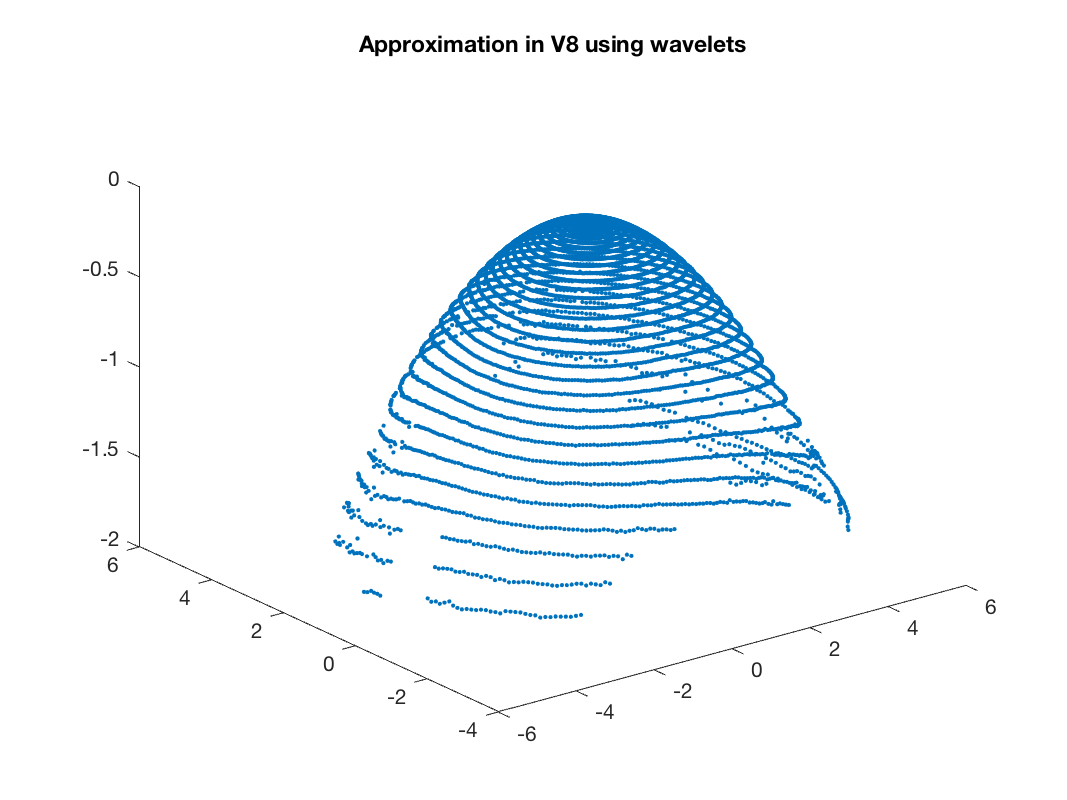}
\includegraphics[width=.45 \textwidth]{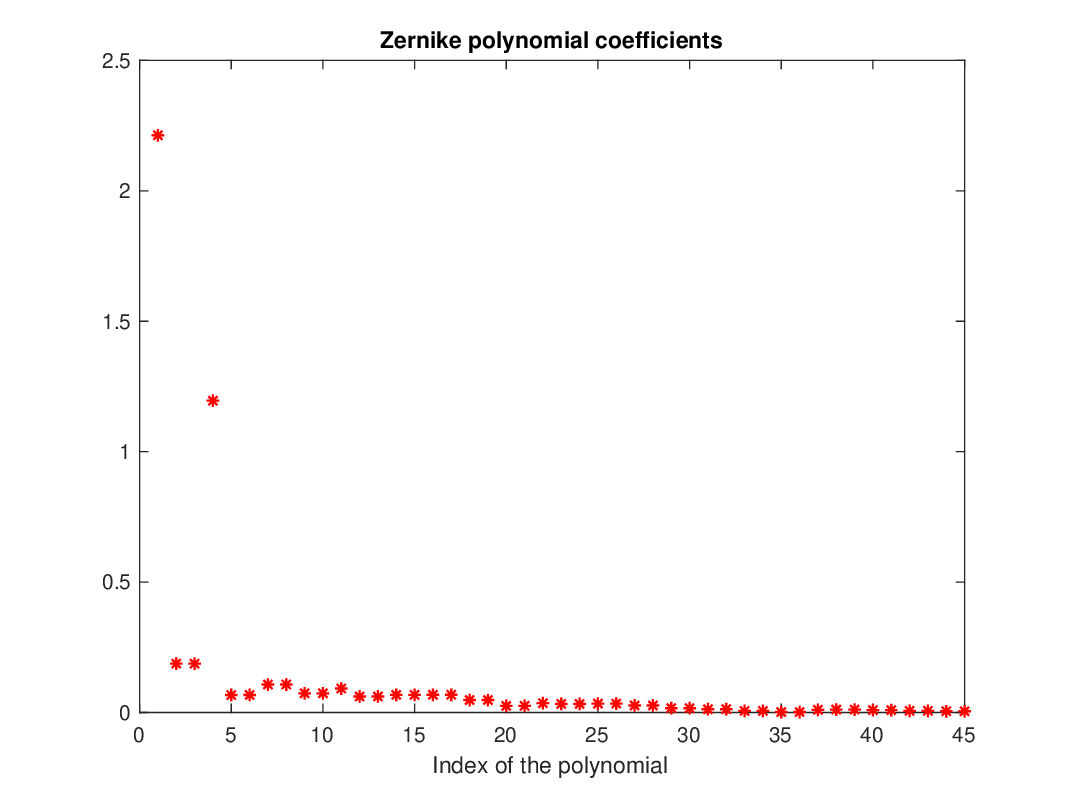}
\includegraphics[width=.45\textwidth]{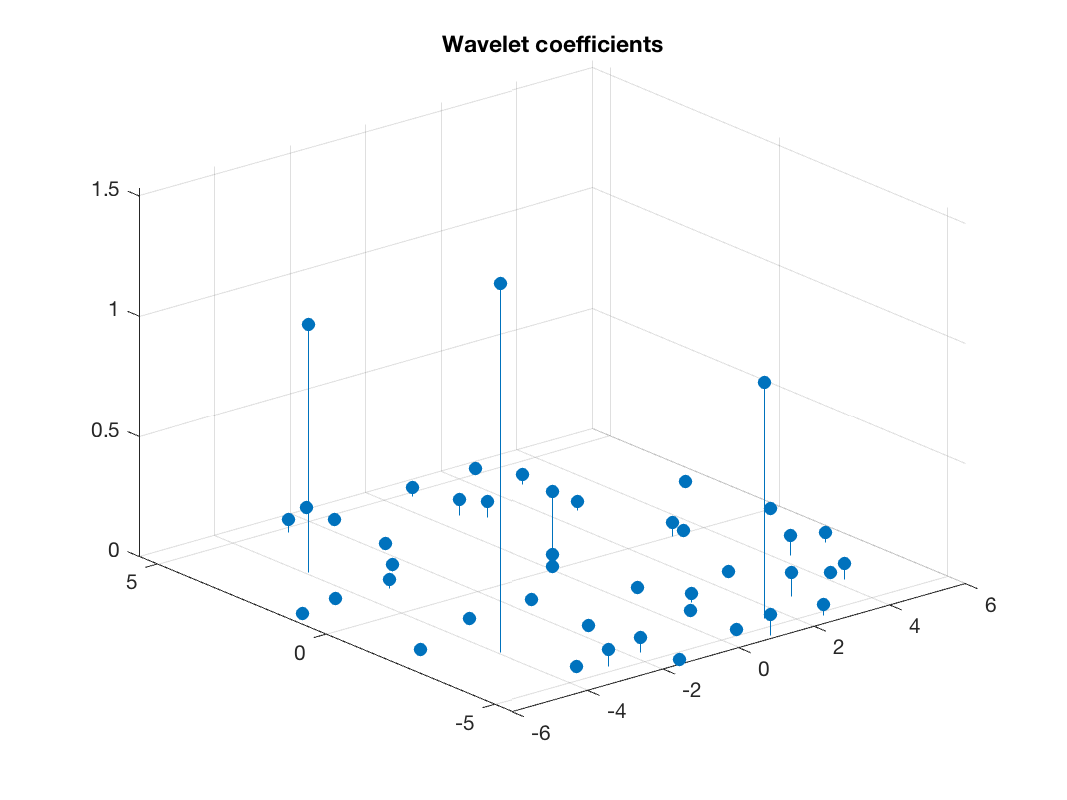}
\includegraphics[width=.45 \textwidth]{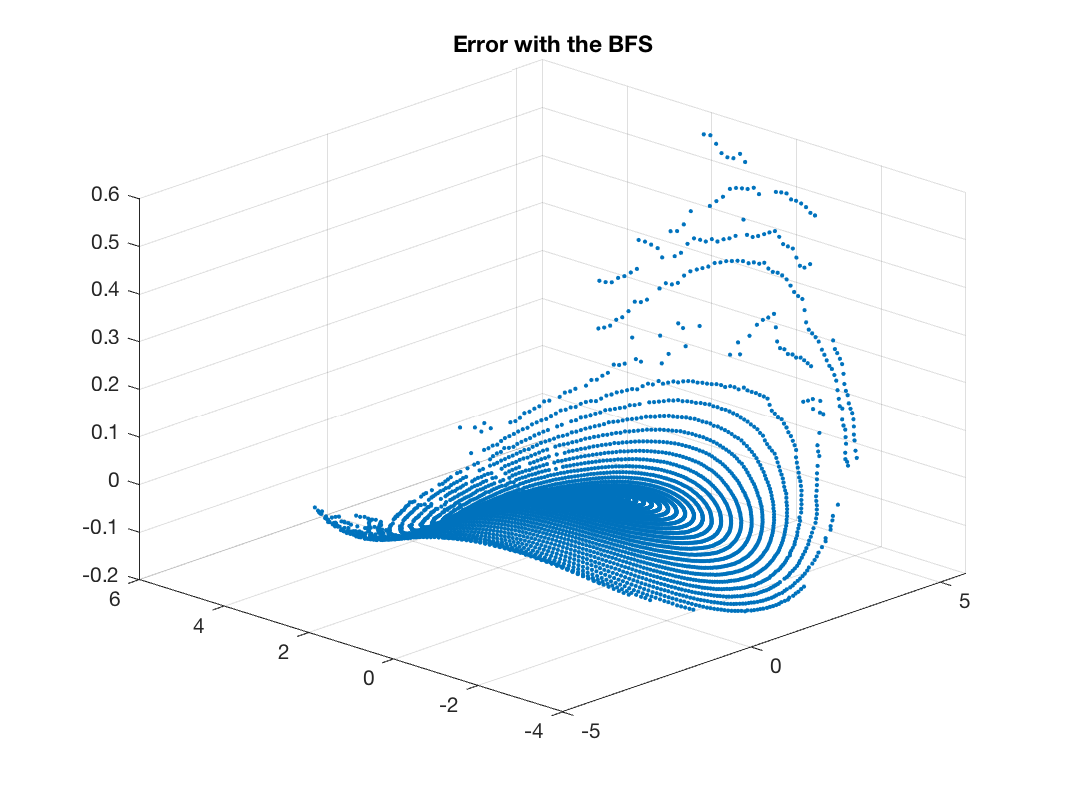}
\caption{Wavelet and Zernike coefficients of the data from a subject with keratoconus of Figure~\ref{FIG:Approximation3}}
\label{FIG:WaveletCoefsKera}
\end{figure}

\begin{figure}[!h]
\centering
\includegraphics[width=.45 \textwidth]{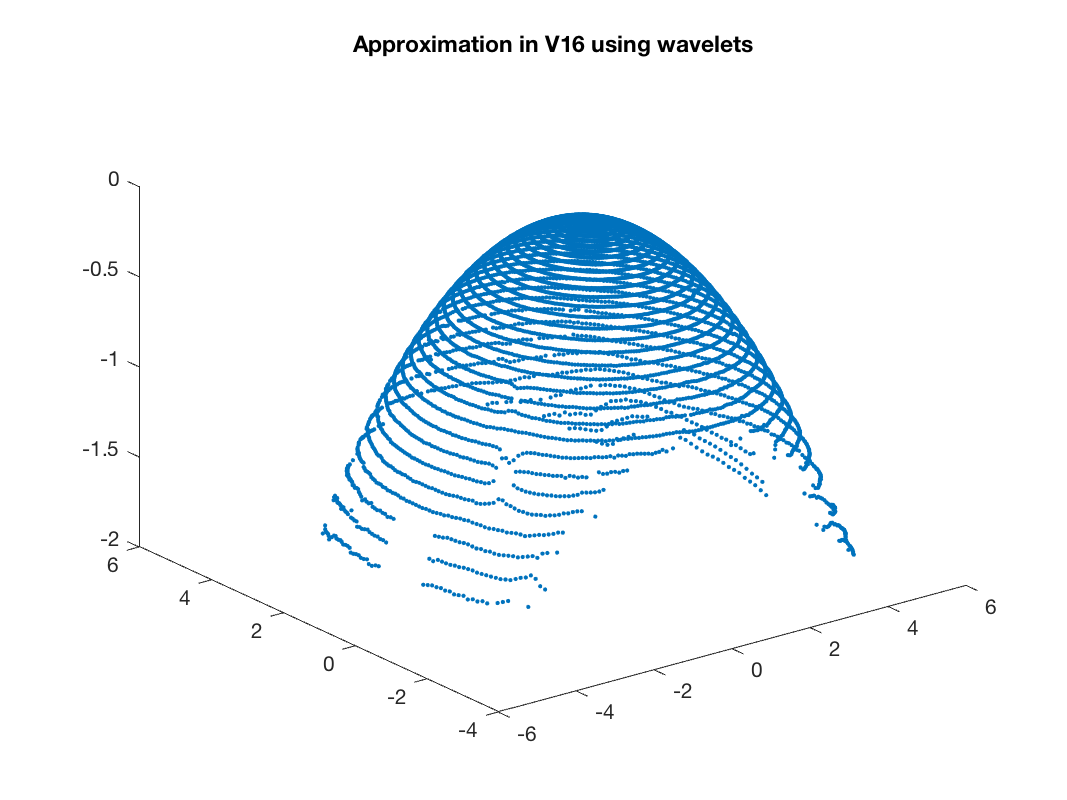}
\includegraphics[width=.45 \textwidth]{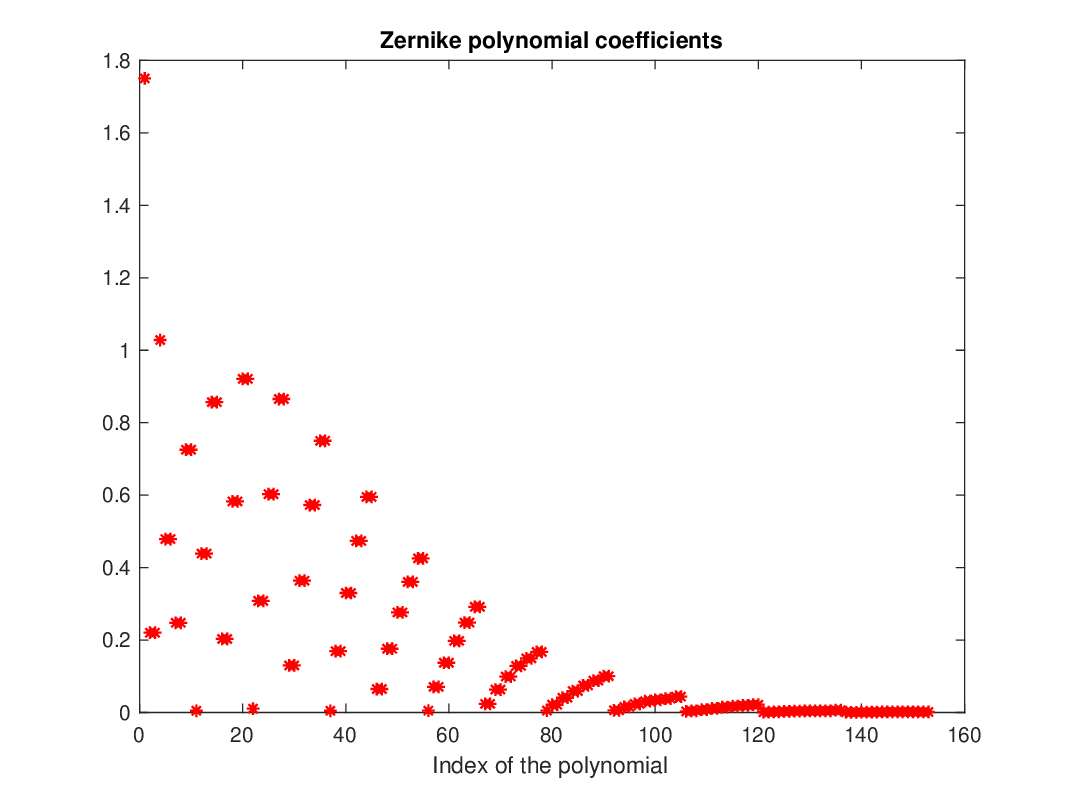}
\includegraphics[width=.45\textwidth]{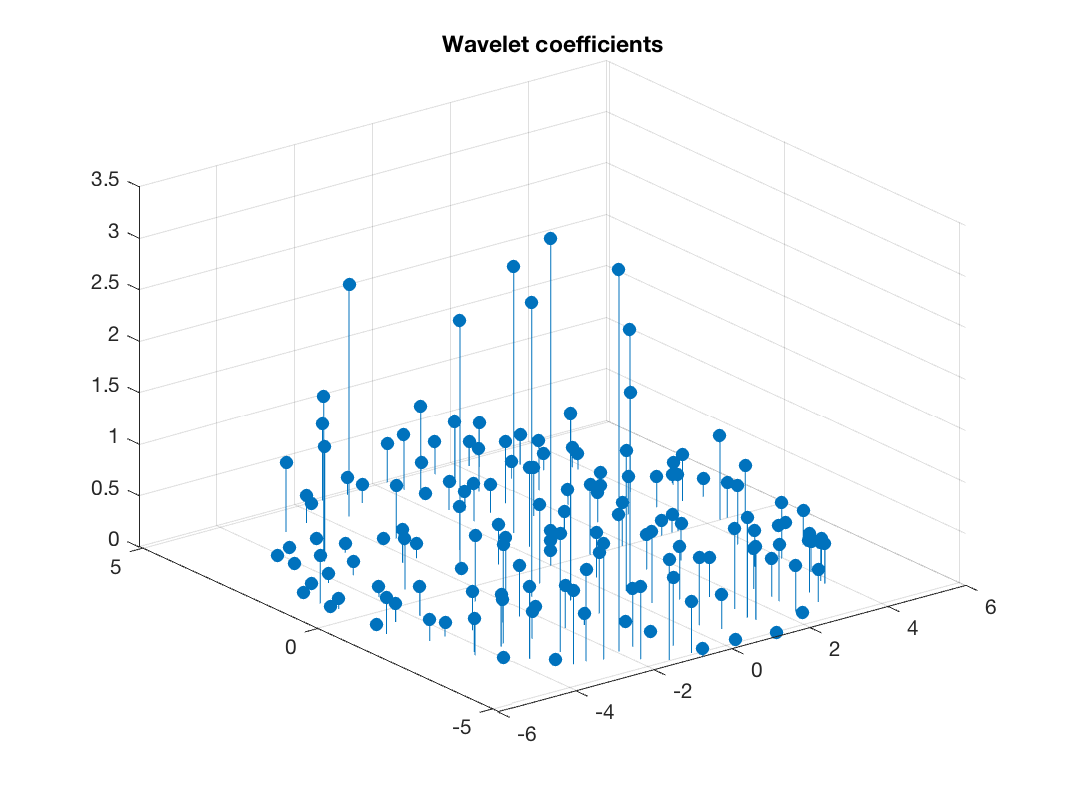}
\includegraphics[width=.45 \textwidth]{ErrorBFSAstig.eps}
\caption{Wavelet and Zernike coefficients of the astigmatism data of Figure~\ref{FIG:Approximation} using $N=16$}
\label{FIG:WaveletCoefsAstigN16}
\end{figure}

\section{Conclusions}
In this paper, it is shown how to represent 2D signals on circular domains using representations with a wavelet structure via kernel polynomials made from Zernike polynomials. The property of a kernel polynomial that makes it localized about the  parameter point and the nature of the Zernike polynomials defined on circular domains play significant roles in the process. The wavelet coefficients in the signal representation will indicate the locations of spatial frequencies in the data. When applied to ophthalmic data in the corneal surface, depending on the magnitude and location of these coefficients, certain aberrations such as astigmatism may be diagnosed. In feature extraction, image processing, and other applications, locations of Zernike wavelet coefficients could be related to the abnormalities of materials and images, a vast field to be explored in future. Our focus here was on development of the theory. Our hope for the future is to obtain more informative numerical results with the help of more powerful computational tools and better point selection for the wavelet functions using numerical methods such as the QR algorithm for Fekete points.

\section{Acknowledgments}
\label{sec:Acknowledge}
 The authors are immensely grateful to Dr.~D.~Robert Iskander for providing the data used for experimental results in this work. The first named author is also grateful to Dr.~Rodrigo Platte (Arizona State University) for useful discussions regarding the content of this paper. The first named author would like to acknowledge support from Simons Foundation under Award No.~709212. Finally, the authors thank the anonymous reviewer for valuable suggestions.

\bibliographystyle{unsrt} 

\begin{thebibliography}{10}

\frenchspacing

\bibitem{Bos1981}
L.~Bos.
\newblock Near optimal location of points for Lagrange interpolation in several variables.
\newblock Ph.D. thesis, University of Toronto, 1981.


\bibitem{Bos2008}
L.~Bos and N.~Levenberg.
\newblock On the calculation of approximate Fekete points: the univariate case.
\newblock Electronic Transactions on Numerical Analysis.
\newblock Vol. 30, 377 -- 397, 2008.

\bibitem{Chr03}
O.~Christensen.
\newblock An Introduction to Frames and Riesz Bases.
\newblock Birkh\"{a}user, Boston, Second Edition, Applied and Numerical Harmonic Analysis series, 2016. 



\bibitem{Cle2015}
C.~Clemente, L.~Pallotta, A.~De Maio, J. J.~Soraghan and A.~Farina.
\newblock A novel algorithm for radar classification based on doppler characteristics exploiting orthogonal Pseudo-Zernike polynomials.
\newblock IEEE Transactions on Aerospace and Electronic Systems. 
\newblock Vol. 51, No. 1, 417 -- 430, January 2015.



\bibitem{FAROKHI2014}
S.~Farokhi, S.M.~Shamsuddin, U.U.~Sheikh, J.~Flusser, M.~Khansari, and K.~Jafari-Khouzani.
\newblock Near infrared face recognition by combining Zernike moments and undecimated discrete wavelet transform. 
\newblock Digital Signal Processing.
\newblock Vol. 31, 13 -- 27, 2014.

\bibitem{Fischer1997}
B.~Fischer and J.~Prestin.
\newblock Wavelets based on orthogonal polynomials.
\newblock Mathematics of Computation.
\newblock Vol. 66, No. 220, 1593 -- 1618, 1997.

\bibitem{GM1992}
K.~Gr\"{o}chenig and W.~Madych.
\newblock Multiresolution analysis, {H}aar bases, and self-similar tilings of $\mathbb{R}^n.$
\newblock IEEE Transactions on Information Theory.
\newblock Vol.~38, 556 -- 568, 1992.

\bibitem{Haar1910}
A.~Haar.
\newblock Zur Theorie der Orthogonalen Functionen-Systeme.
\newblock Math. Ann. 
\newblock Vol.~69, 331 -- 371, 1910. 

\bibitem{Iskander2002}
D.~R.~Iskander, M.~R.~Morelande, M.~J.~Collins, and B.~Davis.
\newblock Modeling of corneal surfaces with radial polynomials.
\newblock IEEE Transactions on Biomedical Engineering.
\newblock Vol.~49, no. 4, 320 -- 328, April 2002.
\newblock doi: 10.1109/10.991159.



\bibitem{Kamal2016}
N.D. ~Mustafa Kamal, N.~Jalil, and H.~Hashim. 
\newblock  The analysis of shape-based, DWT and zernike moments feature extraction techniques for fasterner recognition using 10-fold cross validation multilayer perceptrons.
\newblock Journal of Electrical and Electronic Systems Research (JEESR). 
\newblock Vol. 9 (8), 43 -- 51, 2016. 


\bibitem{Kilgore1996}
T.~Kilgore and J.~Prestin.
\newblock Polynomial wavelets on the interval.
\newblock Constructive Approximation. 
\newblock Vol. 12, 95 -- 110, 1996. 


\bibitem{kinter76} 
E.~Kinter.
\newblock On the mathematical properties of Zernike polynomials.
\newblock Optica Acta.
\newblock Vol. 23, 679 -- 80, 1976.



\bibitem{Madych1992}
W.~Madych.
\newblock Some elementary properties of multiresolution analyses of $L^2(R^n)$
\newblock in Wavelets: {A} {T}utorial in {T}heory and {A}pplications {Vol} II, edited by C. Chui.
\newblock Academic Press, Inc., 259 -- 294, 1992.


\bibitem{Mallat89}
S.~Mallat.
\newblock Multiresolution approximations and wavelet orthonormal bases for $L^2(\mathbb{R})$.
\newblock Trans. Amer. Math. Soc. 
\newblock Vol. 315, 69 -- 87, 1989. 

\bibitem{Meyer}
Y.~Meyer.
\newblock Wavelets: Algorithms \& applications.
\newblock Translated from French and with a foreword by Robert D.~Ryan.
\newblock SIAM, Philadelphia, 1993. 


\bibitem{Nijboer1942}
B.~R.~A.~Nijboer.
\newblock The Diffraction Theory of Aberrations.
\newblock Ph.D. Thesis, University of Groningen, 1942.
\newblock Also, Physica (23), 605--620, 1947.



\bibitem{Lopez2016}
D.~Ramos-L\'{o}pez, M. A. S\'{a}nchez-Granero, and M.~Fern\'{a}ndez-Mart\'{i}nez.
\newblock Optimal sampling patterns for {Z}ernike polynomials.
\newblock Applied Mathematics and Computation.
\newblock Vol.~274, 247 -- 257, 2016.


\bibitem{Sommariva2009}
A.~Sommariva and M.~Vianello.
\newblock Computing approximate Fekete points by QR factorizations of Vandermonde matrices.
\newblock Computers \& Mathematics with Applications.
\newblock Vol.~57 (8), 1324 -- 1336, 2009.


\bibitem{Szego1975} 
G.~Szeg\"{o}.
\newblock Orthogonal Polynomials. 
\newblock American Mathematical Society.
1975.



\bibitem{YuanXu2004}
Y.~Xu.
\newblock Lecture notes on orthogonal polynomials of several variables.
\newblock Advances in the Theory of Special Functions and Orthogonal Polynomials.
\newblock Nova Science Publishers.
\newblock Vol.~2, 135 -- 188, 2004.


\bibitem{Zernike1934}
F.~Zernike.
\newblock Beugungstheorie des schneidenver-fahrens und seiner verbesserten form, der phasenkontrastmethode.
\newblock Physica.
\newblock Vol. 1 (7-12), 689 -- 704, 1934.


\end{thebibliography}

%
\end{document}